\documentclass[12pt]{amsart}

\usepackage[backend=biber, style=ext-alphabetic, sorting=nyt, articlein=false, backref=false, url=false, isbn=false, giveninits=true, maxbibnames=99]{biblatex}
\DeclareFieldFormat{postnote}{#1}
\DeclareFieldFormat{multipostnote}{#1}

\usepackage{amssymb}
\usepackage[foot]{amsaddr}
\usepackage[margin=1in]{geometry}
\usepackage[hidelinks,pdfusetitle]{hyperref}
\usepackage{bm}
\usepackage[Symbol]{upgreek}
\usepackage{mathtools}
\usepackage{xcolor}
\usepackage{enumitem}
\setlist[enumerate, 1]{label=(\arabic*),leftmargin=4em}
\AtBeginEnvironment{thm}{\setlist[enumerate,1]{label=(\roman*),font=\upshape,leftmargin=4em}}
\AtBeginEnvironment{thmint}{\setlist[enumerate,1]{label=(\roman*),font=\upshape,leftmargin=4em}}
\AtBeginEnvironment{lem}{\setlist[enumerate,1]{label=(\roman*),font=\upshape,leftmargin=4em}}
\AtBeginEnvironment{prop}{\setlist[enumerate,1]{label=(\roman*),font=\upshape,leftmargin=4em}}
\AtBeginEnvironment{cor}{\setlist[enumerate,1]{label=(\roman*),font=\upshape,leftmargin=4em}}

\newcommand\myitem[1][]{\item[#1]\refstepcounter{enumi}}

\usepackage{lmodern}
\usepackage[T1]{fontenc}
\usepackage{microtype}

\DeclareMathOperator{\res}{res}

\DeclareMathOperator{\cf}{cf}

\DeclareMathOperator{\tp}{tp}
\DeclareMathOperator{\qftp}{qftp}

\DeclareMathOperator{\h}{h}

\def\d{\operatorname{d}}
\def\dh{\operatorname{dh}}
\def\nl{\operatorname{nl}}
\def\dhl{\operatorname{dhl}}
\def\dalg{\operatorname{dalg}}
\def\ac{\operatorname{ac}}
\def\rc{\operatorname{rc}}

\def\divop{\operatorname{div}}

\def\sm{\operatorname{small}}
\def\codf{\operatorname{codf}}

\def\r{\operatorname{r}}

\def\LE{\operatorname{LE}}

\newcommand{\ring}{\textnormal{ring}}

\newcommand{\deft}[1]{\textbf{\textup{#1}}}
\newcommand{\ca}{\mathcal}
\newcommand{\bb}{\mathbb}
\newcommand{\llp}{(\!(}
\newcommand{\rrp}{)\!)}
\newcommand{\pconv}{\rightsquigarrow}
\newcommand{\prece}{\preccurlyeq}
\newcommand{\succe}{\succcurlyeq}
\newcommand{\x}{\times}
\newcommand{\f}{\mathfrak}

\newcommand{\0}{\emptyset}

\newcommand{\ges}{\geqslant}
\newcommand{\les}{\leqslant}
\newcommand{\N}{\mathbb{N}}
\newcommand{\Z}{\mathbb{Z}}
\newcommand{\Q}{\mathbb{Q}}
\newcommand{\R}{\mathbb{R}}
\newcommand{\T}{\mathbb{T}}

\DeclareFontFamily{U}{fsy}{}
\DeclareFontShape{U}{fsy}{m}{n}{<->s*[.9]psyr}{}
\DeclareSymbolFont{der@m}{U}{fsy}{m}{n}
\DeclareMathSymbol{\der}{\mathord}{der@m}{182}
\DeclareFontFamily{OMS}{smallo}{}
\DeclareFontShape{OMS}{smallo}{m}{n}{<->s*[.65]cmsy10}{}
\DeclareSymbolFont{smallo@m}{OMS}{smallo}{m}{n}
\DeclareMathSymbol{\cao}{\mathord}{smallo@m}{79}

\newtheorem{thmint}{Theorem}

\newtheorem{lem}{Lemma}[section]
\newtheorem{prop}[lem]{Proposition}
\newtheorem{cor}[lem]{Corollary}
\newtheorem{thm}[lem]{Theorem}

\newtheorem*{ac:qe}{Theorem~\ref{ac:qe}}
\newtheorem*{preH:qe}{Theorem~\ref{preH:qe}}
\newtheorem*{preH:modcomp}{Corollary~\ref{preH:modcomp}}

\theoremstyle{definition}

\newtheorem*{defn}{Definition}
\newtheorem*{ass}{Assumption}

\numberwithin{claim}{lem}
\numberwithin{equation}{section}

\makeatletter
\newcommand{\manuallabel}[2]{\def\@currentlabel{#2}\label{#1}}
\makeatother

\usepackage{etoolbox}

\makeatletter
\patchcmd{\@startsection}
  {\@afterindenttrue}
  {\@afterindentfalse}
  {}{}
\makeatother

\title{Model theory of differential-henselian pre-\texorpdfstring{$H$}{H}-fields}
\author{Nigel Pynn-Coates}
\address{Kurt G\"{o}del Research Center, Institute of Mathematics, University of Vienna, Austria}
\email{\href{mailto:nigel.pynn-coates@univie.ac.at}{nigel.pynn-coates@univie.ac.at}}

\begin{document}

\begin{abstract}
Pre-$H$-fields are ordered valued differential fields satisfying some basic axioms coming from transseries and Hardy fields.
We study pre-$H$-fields that are differential-Hensel-Liouville closed, that is, differential-henselian, real closed, and closed under exponential integration, establishing an Ax--Kochen/Ershov theorem for such structures: the theory of a differential-Hensel-Liouville closed pre-$H$-field is determined by the theory of its ordered differential residue field; this result fails if the assumption of closure under exponential integration is dropped.
In a two-sorted setting with one sort for a differential-Hensel-Liouville closed pre-$H$-field and one sort for its ordered differential residue field, we eliminate quantifiers from the pre-$H$-field sort, from which we deduce that the ordered differential residue field is purely stably embedded and if it has NIP, then so does the two-sorted structure.
Similarly, the one-sorted theory of differential-Hensel-Liouville closed pre-$H$-fields with closed ordered differential residue field has quantifier elimination, is the model completion of the theory of pre-$H$-fields with gap~$0$, and is complete, distal, and locally o-minimal.
\end{abstract}
\maketitle

\section{Introduction}

In \cite{adamtt}, Aschenbrenner, van den Dries, and van der Hoeven study the model theory of $\T$ as an ordered valued differential field, where $\T$ is the differential field of logarithmic-exponential transseries constructed in \cite{dmm97}, there denoted by $\R\llp t \rrp^{\LE}$.
As part of this project, Aschenbrenner and van den Dries introduced in \cite{ad-hf} the elementary class of pre-$H$-fields, which are ordered valued differential fields satisfying conditions capturing some of the basic interactions between the ordering, valuation, and derivation in transseries and Hardy fields.
As such, the model theory of transseries leads naturally to the model theory of pre-$H$-fields.
Among the many results of \cite{adamtt}, for which the three received the 2018 Karp Prize from the Association for Symbolic Logic, is an effective axiomatization of the theory of $\T$ as a complete theory extending the theory of pre-$H$-fields.

But there are pre-$H$-fields satisfying different elementary conditions than $\T$, including $\d$-henselian ones (we use ``$\d$'' to abbreviate ``differential'' or ``differentially'', as appropriate), which are the focus of this paper.
This is a generalization of henselianity of a valued field to the class of valued differential fields with small derivation, a strong form of continuity of the derivation with respect to the valuation topology.
To describe a family of completions of the theory of pre-$H$-fields, we call a pre-$H$-field \emph{$\d$-Hensel-Liouville closed} if it is $\d$-henselian, real closed, and closed under exponential integration (i.e., for each $f$ there is $y \neq 0$ such that $y'/y = f$).
Given a complete theory $T$ of ordered differential fields, the theory of $\d$-Hensel-Liouville closed pre-$H$-fields whose ordered differential residue fields satisfy $T$ is complete.
Equivalently:
\begin{thmint}\label{int:ake}
Let $K_1$ and $K_2$ be $\d$-Hensel-Liouville closed pre-$H$-fields with ordered differential residue fields $\bm k_1$ and $\bm k_2$.
Then $K_1 \equiv K_2$ if and only if $\bm k_1 \equiv \bm k_2$.
\end{thmint}

More precisely, the pre-$H$-field language is $\{+, -, \cdot, 0, 1, \der, \les, \prece\}$, where, if $\ca O$ is the valuation ring, $f \prece g$ means $f \in g\ca O$.
Each residue field is equipped with the ordering and derivation induced by those of its pre-$H$-field and construed as a structure in the language $\{+, -, \cdot, 0, 1, \der, \les\}$.

Theorem~\ref{int:ake} is a result in the spirit of Ax--Kochen \cite{ak3} and Ershov \cite{ershov}, who showed that the theory of a henselian valued field of equicharacteristic~$0$ is determined by the theory of its residue field and (ordered) value group.
The value groups of $K_1$ and $K_2$ do not appear in Theorem~\ref{int:ake} because under its assumptions they are are divisible ordered abelian groups, and hence elementarily equivalent.
Moreover, they are even elementarily equivalent as asymptotic couples, as defined by Rosenlicht \cite{rosen-dval}, which is to say after being expanded by the map induced by logarithmic differentiation (see \S\ref{sec:prelim}).
In addition to the axioms for divisible ordered abelian groups, the complete axiomatization is given by four universal axioms of asymptotic couples together with the statement that the logarithmic derivative map is surjective onto the set of negative elements of the group; this is Corollary~\ref{ac:complete} (where these structures are called ``gap-closed $H$-asymptotic couples'').

Analogous AKE results for (unordered) valued differential fields have been established earlier for $\d$-henselian monotone fields: first for those with many constants by Scanlon \cite{scanlon} and then in the general monotone case by Hakobyan \cite{hakobyan}.
The interaction between the valuation and the derivation in a pre-$H$-field is opposite to the condition of monotonicity, and, indeed, $\d$-henselian pre-$H$-fields are never monotone.

To which pre-$H$-fields does Theorem~\ref{int:ake} apply?
Certainly not $\T$, which is not $\d$-henselian.\footnote{Instead, $\T$ satisfies a strong eventual form of $\d$-henselianity called newtonianity; see \cite[Chapter~14]{adamtt}.}
To construct an example, start with an $\aleph_0$-saturated elementary extension $\T^*$ of $\T$.
Although every element of $\T$ is exponentially bounded in the sense that it is bounded in absolute value by some finite iterate of the exponential, $\T^*$ contains transexponential (i.e., not exponentially bounded) elements.
Enlarge the valuation ring $\ca O_{\T^*}$ of $\T^*$ to the set $\dot{\ca O}_{\T^*}$ of exponentially bounded elements of $\T^*$.
With this coarsened valuation, $(\T^*, \dot{\ca O}_{\T^*})$ is a $\d$-Hensel-Liouville closed pre-$H$-field whose valuation only distinguishes transexponentially different elements, and
its ordered differential residue field $\res(\T^*, \dot{\ca O}_{\T^*})$ is a model of the theory of $\T$.
This coarsening decomposes $\T^*$ into a transexponential part, $(\T^*, \dot{\ca O}_{\T^*})$, whose model theory is the subject of this paper, and an exponentially bounded part, $\res(\T^*, \dot{\ca O}_{\T^*})$, whose model theory we understand by \cite{adamtt}.
The same coarsening procedure applied to a maximal Hardy field, which contains transexponential elements \cite{bosher-transexpHf}, yields a similar example.

Theorem~\ref{int:ake} follows easily from Theorem~\ref{eqthm}, which establishes a back-and-forth system in a two-sorted setting with a sort for a pre-$H$-field and a sort for its ordered differential residue field, connected by a binary residue map.
This back-and-forth system allows us to eliminate quantifiers from the pre-$H$-field sort in Theorem~\ref{relQE}, which is used to show the following two results.
(In fact, we get more precise two-sorted versions.)
\begin{thmint}\label{int:stabembednip}
Let $K$ be a $\d$-Hensel-Liouville closed pre-$H$-field with ordered differential residue field~$\bm k$.
\begin{enumerate}
    \item Every subset of $\bm k^n$ definable in $K$, with parameters, is definable in the ordered differential field $\bm k$ with parameters from~$\bm k$.
    \item If the ordered differential field $\bm k$ has NIP, then $K$ has NIP.
\end{enumerate}
\end{thmint}

Another of the central results of \cite{adamtt} is that the class of existentially closed pre-$H$-fields is elementary and axiomatized by a theory $T^{\nl}$; note that $T^{\nl}_{\sm} = T^{\nl} + \text{``small derivation''}$ axiomatizes the theory of $\T$.
Equivalently, $T^{\nl}$ is the model companion of the theory of pre-$H$-fields, that is, $T^{\nl}$ is model complete and every pre-$H$-field extends to a model of $T^{\nl}$.
Also, $T^{\nl}_{\sm}$ is the model companion of the theory of $H$-fields with small derivation, where an \emph{$H$-field} is a pre-$H$-field whose valuation ring is the convex hull of its constant field.
But a pre-$H$-field with small derivation extends to a model of $T^{\nl}_{\sm}$ if and only if the derivation induced on its residue field is trivial.
This raises the question of whether the theory of pre-$H$-fields with nontrivial induced derivation on their residue field, which includes all $\d$-henselian pre-$H$-fields, has a model companion.

We give a positive answer for the broader class of pre-$H$-fields with gap~$0$.
In defining this, we use the notation $f \asymp g$ when $f \prece g$ and $g \prece f$, and $f \prec g$ when $f \prece g$ and $f \not\asymp g$.
A pre-$H$-field has \emph{small derivation} if $f' \prec 1$ for all $f \prec 1$ and \emph{gap~$0$} if it has small derivation and every $f \succ 1$ satisfies $f' \succ f$.
If a pre-$H$-field with gap~$0$ is closed under exponential integration (or even if its ordered differential residue field is), then any $f \succ 1$ is ``transexponential'' in the sense that $f \succ e_n$ for each $n$, where $e_n$ is an $n$-th iterated exponential integral of $1$.
To describe the axiomatization of the model completion, we recall that the theory of ordered differential fields, with no assumption on the interaction between the ordering and the derivation, has a model completion, the theory of closed ordered differential fields.
This theory also has quantifier elimination and is complete \cite{singer-codf}.
Then:
\begin{thmint}\label{int:modcomp}
The theory $T^{\dhl}_{\codf}$ 
of $\d$-Hensel-Liouville closed pre-$H$-fields with closed ordered differential residue field
is the model completion of the theory of pre-$H$-fields with gap~$0$.
\end{thmint}

In fact, we prove a more general two-sorted version where other theories of the residue field are permitted, as in the previous two results.
By Theorem~\ref{int:ake}, $T^{\dhl}_{\codf}$ is complete.
This theory has quantifier elimination, so it is distal by \cite{acgz-distal} and locally o-minimal.
The one-sorted results are collected in Theorem~\ref{preH:qe}.

The assumptions of ``real closed'' and ``exponential integration'' are necessary in Theorem~\ref{int:modcomp}, but what of their appearance in Theorem~\ref{int:ake}?
Even in the monotone setting, it is not true that the theory of a $\d$-henselian field is determined by the theory of its differential residue field and its value group.
Hakobyan provides an example of two $\d$-henselian monotone fields that are not elementarily equivalent but have isomorphic value groups and differential residue fields \cite[Example after Corollary~4.2]{hakobyan}.
To remedy this, in his monotone AKE theorem there is an additive map from the value group to the residue field.
Hence any AKE theorem for $\d$-henselian fields requires either additional assumptions or extra structure on the value group or differential residue field.
While it may be possible to remove the assumption that $K_1$ and $K_2$ are real closed in Theorem~\ref{int:ake} at the expense of incorporating their asymptotic couples, the next result shows that closure under exponential integration is necessary.
\begin{thmint}\label{int:ex}
There exist $\d$-henselian, real closed pre-$H$-fields $K_1$ and $K_2$ with ordered differential residue fields $\bm k_1$ and $\bm k_2$ such that $\bm k_1 \cong \bm k_2$ but $K_1 \not\equiv K_2$.
\end{thmint}
More precisely, one of $K_1$ and $K_2$ is closed under exponential integration but the other is not.
Also, their asymptotic couples are elementarily equivalent, so any three-sorted improvement of Theorem~\ref{int:ake} would still require some assumption on the behaviour of logarithmic derivatives.

\subsection{Structure of the paper}
Section~\ref{sec:prelim} contains basic definitions, notation, and remarks, which we keep close to \cite{adamtt}.
With the goal of back-and-forth arguments in mind, in subsequent sections we consider extensions controlled by the residue field, extensions controlled by the asymptotic couple, and those that involve adjoining exponential integrals.

In \S\ref{sec:preH:res}, we show how to extend embeddings of ordered valued differential fields by first extending the residue field.

We study extensions controlled by the asymptotic couple in \S\ref{sec:extac}, starting by studying asymptotic couples as structures in their own right in \S\ref{sec:ac-small}.
We isolate the model completion of the theory of $H$-asymptotic couples with gap~$0$ and prove that this theory has quantifier elimination in Theorem~\ref{ac:qe}.
In the rest of the section, this result, or rather its consequence Corollary~\ref{ac:9.9.2}, is used to study extensions of $\d$-henselian pre-$H$-fields whose asymptotic couples are existentially closed.

The short \S\ref{sec:rcc} deals with extending the constant field and plays no role in the main results; it is used only to strengthen the statement of Theorem~\ref{rcexpintclosure} in the next section.

Section~\ref{sec:dhl} proves the most difficult embedding result of the paper, Theorem~\ref{dhlclosure}: the existence of $\d$-Hensel-Liouville closures, which are are extensions that are $\d$-henselian, real closed, and closed under exponential integration, and that satisfy a semi-universal property.
Uniqueness of $\d$-Hensel-Liouville closures is Corollary~\ref{dhlclosuremin}.

Turning to the main results, in \S\ref{sec:main}, first some embedding results from previous sections are combined in the key embedding lemma.
Next we establish the two-sorted back-and-forth system, in Theorem~\ref{eqthm}.
The AKE theorem, Theorem~\ref{int:ake} (Corollary~\ref{ake}), follows immediately.
The relative quantifier elimination is Theorem~\ref{relQE}, from which the stable embeddedness of the residue field (the first part of Theorem~\ref{int:stabembednip}) follows immediately.
The two-sorted model companion result is Corollary~\ref{multimodcompa}, and the NIP transfer principle (the second part of Theorem~\ref{int:stabembednip}) is Proposition~\ref{NIPtransfer}.
One-sorted results for $T^{\dhl}_{\codf}$, including Theorem~\ref{int:modcomp} but also quantifier elimination, distality, and local o-minimality, are collected in Theorem~\ref{preH:qe}.

Section~\ref{sec:examples} documents two examples of $\d$-henselian pre-$H$-fields.
The first, in \S\ref{sec:transexppreH}, provides a fuller account of the example described in the introduction of a $\d$-Hensel-Liouville closed pre-$H$-field whose residue field is a model of $T^{\nl}_{\sm}$, which arises from a transexponential extension of $\T$.
This example motivates the earlier two-sorted results, in which the ordered differential residue field could have additional structure.
The second, in \S\ref{sec:expintnecessary}, establishes Theorem~\ref{int:ex} (Corollary~\ref{ex:notelem}), showing that the assumption of closure under exponential integration in the AKE theorem cannot be dropped.
As part of the proof, we show that for every ordered differential field $\bm k$ satisfying three obviously necessary conditions, there is a $\d$-Hensel-Liouville closed pre-$H$-field with ordered differential residue field isomorphic to~$\bm k$.

\section{Preliminaries}\label{sec:prelim}
We let $m$, $n$, and $r$ range over $\N = \{0, 1, 2, \dots\}$ and $\rho$, $\lambda$, and $\mu$ be ordinals.
The main objects of this paper are kinds of ordered valued differential fields; all fields in this paper are assumed to be of characteristic $0$.
A \deft{valued field} is a field $K$ equipped with a surjective map $v \colon K^{\x} \to \Gamma$, where $\Gamma$ is a (totally) ordered abelian group, satisfying for $f, g \in K^{\x}$:
\begin{enumerate}[label=(V\arabic*)]
	\item $v(fg) = v(f)+v(g)$;
	\item\label{v2} $v(f+g) \ges \min\{v(f), v(g)\}$ whenever $f+g \neq 0$.
\end{enumerate}
We also impose the condition that $v(\Q^{\x})=\{0\}$, i.e., that $K$ has equicharacteristic~$0$.
A \deft{differential field} is a field $K$ equipped with a \deft{derivation} $\der \colon K \to K$, which satisfies for $f, g \in K$:
\begin{enumerate}[label=(D\arabic*)]
	\item $\der(f+g)=\der(f)+\der(g)$;
	\item $\der(fg)=f\der(g)+g\der(f)$.
\end{enumerate}

Let $K$ be a valued field.
We add a new symbol $\infty$ to $\Gamma$ and extend the addition and ordering to $\Gamma_\infty \coloneqq \Gamma \cup \{\infty\}$ by $\infty+\gamma=\gamma+\infty=\infty$ and $\infty>\gamma$ for all $\gamma \in \Gamma$.
This allows us to extend $v$ to $K$ by setting $v(0) \coloneqq \infty$.
We often use the following more intuitive notation:
\[\begin{array}{lc}
f \prece g\ \Leftrightarrow\ v(f)\ges v(g),\qquad f \prec g\ \Leftrightarrow\ v(f)>v(g),\\
f \asymp g\ \Leftrightarrow\ v(f)=v(g),\qquad  
  f\sim g\ \Leftrightarrow\ f-g \prec g.
\end{array}\]
The relation $\prece$ is called a \deft{dominance relation}.
Both $\asymp$ and $\sim$ are equivalence relations on $K$ and $K^{\x}$ respectively, with a consequence of \ref{v2} being that if $f \sim g$, then $f \asymp g$.
We set $\ca O \coloneqq \{ f \in K : f \prece 1\}$ and call it the \deft{valuation ring} of $K$.
It has a (unique) maximal ideal $\cao \coloneqq \{ f \in K : f \prec 1\}$, and we call $\res(K) \coloneqq \ca O/\cao$ the \deft{residue field} of $K$, usually denoted by $\bm k$.
We also let $\overline{a}$ or $\res(a)$ denote the image of $a \in \ca O$ under the natural map to $\bm k$.
For another valued field $L$, we denote these objects by $\ca O_L$, $\Gamma_L$, $\bm k_L$, etc.

Let $K$ be a differential field.
For $f \in K$, we often write $f'$ for $\der(f)$ if the derivation is clear from context and set $f^\dagger \coloneqq f'/f$ if $f \neq 0$, the logarithmic derivative of $f$.
We say that $K$ has \deft{exponential integration} if $(K^{\x})^\dagger = K$.
The \deft{field of constants} of $K$ is $C \coloneqq \{ f \in K : f'=0 \}$.
For another differential field $L$, we denote this object by $C_L$.
We let $K\{Y\} \coloneqq K[Y, Y', Y'', \dots]$ be the ring of differential polynomials over $K$ and set $K\{Y\}^{\neq} \coloneqq K\{Y\} \setminus \{0\}$.
For $P \in K\{Y\}^{\neq}$, the \deft{order} of $P$ is the smallest $r$ such that $P \in K[Y, Y', \dots, Y^{(r)}]$.
We extend the derivation of $K$ to $K\{Y\}$ in the natural way.
If $L$ is a differential field extension of $K$ and $a \in L$, then $K \langle a \rangle$ denotes the differential subfield of $L$ generated by $a$ over $K$.
If $K$ is additionally a valued field, we extend $v$ to $K\{Y\}$ by setting $v(P)$ to be the minimum valuation of the coefficients of $P$ and thus also extend the relations $\prece$, $\prec$, $\asymp$, and $\sim$ to $K\{Y\}$. 

\begin{ass}
Suppose for the rest of the paper that $K$ is (at least) a valued differential field, unless stated otherwise.
\end{ass}
Relating the valuation and the derivation, we impose throughout most of this paper the condition that $K$ has \deft{small derivation}, which means that $\der\cao \subseteq \cao$.
If $K$ has small derivation, $\der\ca O \subseteq \ca O$ \cite[Lemma~4.4.2]{adamtt}, so $\der$ induces a derivation on $\bm k$, 
and we always construe $\bm k$ as a differential field with this induced derivation.
In contrast to the main results of \cite{adamtt}, we are interested in the case that it is nontrivial.
We call $\bm k$ \deft{linearly surjective} if for all $a_0, \dots, a_r \in \bm k$ with $a_r \neq 0$, the equation $1 + a_0y + a_1y' + \dots + a_r y^{(r)} = 0$ has a solution in $\bm k$.
We call $K$ \deft{differential-henselian} (\deft{$\d$-henselian} for short) if $K$ has small derivation and:
\begin{enumerate}[label=(DH\arabic*)]
	\item $\bm k$ is linearly surjective;
	\item whenever $P \in \ca O\{Y\}$ of order $r$ satisfies \[P(0)\ \prec\ 1\qquad \text{and} \qquad \sum_{n=0}^r \frac{\partial P}{\partial Y^{(n)}}(0)Y^{(n)}\ \asymp\ 1,\] there is $y \prec 1$ in $K$ with $P(y) = 0$.
\end{enumerate}
Differential-henselianity, a generalization of henselianity to valued differential fields with small derivation, was introduced in \cite{scanlon} and studied more systematically in \cite{adamtt}.

Here is another relation between the valuation and derivation fundamental to this setting.
We call $K$ \deft{asymptotic} if $f \prec g \iff f' \prec g'$ for all nonzero $f, g \in \cao$.
In the rest of this paragraph, suppose that $K$ is asymptotic.
Note that $C \subseteq \ca O$.
Also, if $f, g \in K^{\x}$ satisfy $f \sim g \not\asymp 1$, then $f' \sim g'$ and $f^\dagger \sim g^\dagger$.
Similarly, for $g \in K^{\x}$ with $g \not\asymp 1$, $v(g^\dagger)$ and $v(g')$ depend only on $vg$ and not on $g$, so for $\gamma = vg$ we set $\gamma^\dagger \coloneqq v(g^\dagger)$ and $\gamma' \coloneqq v(g')$; note that $\gamma^\dagger=\gamma'-\gamma$.
For any ordered abelian group $G$, set $G^{\neq} \coloneqq G \setminus \{0\}$, $G^< \coloneqq \{g \in G : g<0\}$, and $G^> \coloneqq \{g \in G : g>0\}$.
Thus logarithmic differentiation induces a map
\begin{align*}
\psi \colon \Gamma^{\neq} &\to \Gamma\\
\gamma &\mapsto \gamma^\dagger.
\end{align*}
We call $(\Gamma, \psi)$ the \deft{asymptotic couple} of $K$; such structures were introduced by Rosenlicht \cite{rosen-dval}, and more about them can be found in \cite[\S6.5 and \S9.2]{adamtt}.
We set $\Psi \coloneqq \psi(\Gamma^{\neq})$, and if we need to indicate the dependence on $K$, we denote this by $\Psi_K$.
When convenient, we extend $\psi$ to a map $\psi \colon \Gamma_{\infty} \to \Gamma_{\infty}$ by setting $\psi(0)=\psi(\infty) \coloneqq \infty$.

Here are two additional properties an asymptotic $K$ can have.
We say that $K$ is \deft{$H$-asymptotic} or of \deft{$H$-type} if, for all $f, g \in K^{\x}$ satisfying $f \prece g \prec 1$, we have $f^\dagger \succe g^\dagger$.
We say that $K$ has \deft{gap~$0$} if it has small derivation and $f^\dagger \succ 1$ for all $f \in K^{\x}$ with $f \prec 1$.
Properties of $K$ are reflected in its asymptotic couple, but conversely some properties of $K$ are determined by $(\Gamma, \psi)$.
For instance, $K$ has small derivation if and only if $(\Gamma^>)' \subseteq \Gamma^>$, and $K$ has gap~$0$ if and only if $(\Gamma^>)' \subseteq \Gamma^>$ and $\Psi \subseteq \Gamma^<$.
There can be at most one $\beta \in \Gamma$ with $\Psi < \beta < (\Gamma^>)'$ \cite[Corollary~9.2.4]{adamtt}, so $K$ having gap~$0$ means that $\beta=0$ separates $\Psi$ and $(\Gamma^>)'$ in this way.
Similarly, $K$ is of $H$-type just in case $\psi(\alpha)\les\psi(\beta)$ whenever $\alpha\les\beta<0$ in $\Gamma$.
For more on asymptotic fields, see \cite[Chapter~9]{adamtt}.

By the above, all asymptotic fields with gap~$0$ are clearly pre-$\d$-valued, where we say that $K$ is \deft{pre-differential-valued} (\deft{pre-$\d$-valued} for short) if:
\begin{enumerate}[label=(PDV)]
    \item for all $f, g \in K^{\x}$ with $f \prece 1$ and $g \prec 1$, we have $f' \prec g^\dagger$.
\end{enumerate}
More on pre-$\d$-valued fields can be found in \cite[\S10.1]{adamtt}, including a characterization of them as those $K$ that satisfy $C \subseteq \ca O$ and a valuation-theoretic analogue of l'H\^{o}pital's Rule \cite[Lemma~10.1.4]{adamtt}.
In the rest of this paragraph, let $K$ be pre-$\d$-valued.
If $K$ has small derivation and the derivation induced on $\bm k$ is nontrivial (for instance, if $K$ is $\d$-henselian), then $K$ has gap~$0$. 
Having gap~$0$ is inherited by valued differential subfields, and thus if $K$ has a $\d$-henselian valued differential field extension, then $K$ has gap~$0$.
These remarks apply in particular to pre-$H$-fields with gap~$0$, which we now define.

This paper is primarily concerned with certain ordered pre-$\d$-valued fields called pre-$H$-fields.
Here, $K$ is an \deft{ordered valued differential field} if, in addition to its valuation and derivation, $K$ is equipped with a (total) ordering $\les$ making it an ordered field.
Relating the ordering, valuation, and derivation, we call $K$ a \deft{pre-$H$-field} if:
\begin{enumerate}[label=(PH\arabic*)]
    \item\label{ph1} $K$ is pre-$\d$-valued;
    \item\label{ph2} $\ca O$ is convex (with respect to $\les$);
    \item\label{ph3} for all $f \in K$, if $f > \ca O$, then $f'>0$.
\end{enumerate}
Condition \ref{ph2} holds if and only if $\cao$ is convex, which holds if and only if $\cao \subseteq (-1,1)$. 
Hence if \ref{ph2} holds, then $\les$ induces an ordering on $\bm k$ making it an ordered field.
We thus construe the residue fields of pre-$H$-fields with small derivation as ordered differential fields.
By \cite[Lemma~10.5.2(ii)]{adamtt}, pre-$H$-fields are of $H$-type.

We now discuss extensions of valued differential fields and ordered valued differential fields.
Given an extension $L$ of $K$, we identify $\Gamma$ with an ordered subgroup of $\Gamma_L$ and $\bm k$ with a subfield of $\bm k_L$ in the obvious way.
Here and in general we use the word \deft{extension} as follows: if $F$ is a valued differential field, ``extension of $F$'' means ``valued differential field extension of $F$''; if $F$ is an ordered valued differential field, ``extension of $F$'' means ``ordered valued differential field extension of $F$''; etc.
``Embedding'', ``isomorphic'', and ``isomorphism'' are used similarly.
Where there is particular danger of confusion, we are explicit.

An important class of extensions are the immediate extensions:
We say that an extension $L$ of $K$ is \deft{immediate} if $\Gamma_L = \Gamma$ and $\bm k_L = \bm k$; equivalently, for every $b \in L^{\x}$ there is $a \in K^{\x}$ such that $b \sim a$.
If $K$ is a pre-$H$-field and $L$ is an immediate valued differential field extension of $K$ that is asymptotic, then $L$ can be given an ordering making it a pre-$H$-field extension of $K$; in fact, this is the unique ordering with respect to which $\ca O_L$ is convex, and hence any valued differential field embedding of $L$ into a pre-$H$-field extension $M$ of $K$ is automatically an ordered valued differential field embedding \cite[Lemma~10.5.8]{adamtt}.

\section{Extensions controlled by the residue field}\label{sec:preH:res}
In this section we prove Lemma~\ref{lem:ordresfldext}, which yields ordered variants of results from \cite[\S6.3 and \S7.1]{adamtt} for use in \S\ref{sec:main}.
In the first lemma, $K$ need not be equipped with a derivation.
\begin{lem}\label{lem:ordresfldext}
Suppose that $K$ is an ordered valued field with convex valuation ring.
Let $L$ be a valued field extension of $K$ with $\Gamma_L=\Gamma$ and suppose that $\bm k_L$ is an \emph{ordered} field extension of $\bm k$.
Then there exists a unique ordering on $L$ making it an ordered field extension of $K$ with convex valuation ring such that the induced ordering on $\bm k_L$ agrees with the given one.

If $K$ is a pre-$H$-field with gap~$0$ and $L$ is moreover a differential field extension of $K$ with small derivation, then $L$ is also a pre-$H$-field with gap~$0$.
\end{lem}
\begin{proof}
Suppose that $L$ is equipped with an ordering making it an ordered field extension of $K$ with convex valuation ring such that the induced ordering on $\bm k_L$ agrees with the given one.
Let $a \in L^{\x}$.
Then $a=su$ with $s \in K^>$ and $u \asymp 1$ in $L$.
Since $u>0 \iff \overline{u}>0$, there is at most one such ordering on $L$, and this also shows how to define the ordering on $L$: $a>0 \iff \overline{u}>0$.
This is independent of the choice of $s$ and $u$: If $a=tu_1$ with $t \in K^>$ and $u_1 \asymp 1$ in $L$, then $u_1 = st^{-1}u$ and $\overline{u_1} = \overline{st^{-1}}\cdot\overline{u}$, so $\overline{u_1}>0 \iff \overline{u}>0$.
Obviously, $a>0$ or $-a>0$.

Next, assume that $a, b \in L^{>}$; we will show that $a+b>0$ and $ab>0$.
Then $a=su_1$ and $b=tu_2$ with $s, t \in K^>$ and $u_1, u_2 \asymp 1$ in $L^>$.
Without loss of generality, $s \prece t$, so $a+b=t(st^{-1}u_1+u_2)$ and 
\[
\overline{st^{-1}u_1+u_2}\ =\ \overline{st^{-1}}\cdot\overline{u_1}+\overline{u_2}\ >\ 0.
\]
Thus $st^{-1}u_1+u_2 \asymp 1$ and $a+b>0$.
Also, $ab = stu_1u_2$ with $\overline{u_1u_2}=\overline{u_1}\cdot\overline{u_2}>0$, so $ab>0$. 
Similarly, $a^2>0$.
Thus we have defined an ordering on $L$ making it an ordered field extension of $K$.
Obviously, if $a \prec 1$, then $-1<a<1$, so the valuation ring of $L$ is convex with respect to this ordering, and by construction it induces the given ordering on~$\bm k_L$.

Finally, suppose that $K$ is a pre-$H$-field with gap~$0$.
Let $a,b \in L^{\x}$ with $a \prece 1$ and $b \not\asymp 1$.
Since $L$ has small derivation, we have $a' \prece 1$.
Write $b=su$ with $s \in K^>$ and $u \asymp 1$ in $L$, so then $b^{\dagger} = s^{\dagger}+u^{\dagger}\sim s^{\dagger} \succ 1$, since $K$ has gap~$0$.
Thus $a' \prec b^{\dagger}$, showing that $L$ is pre-$\d$-valued.
It has the same asymptotic couple as $K$, so still has gap~$0$.
To see that $L$ is a pre-$H$-field, it remains to check \ref{ph3}:
If $b>\ca O_L$, then likewise $b'\sim s'u > 0$.
\end{proof}

\begin{lem}\label{adh7.1.4o}
Suppose that $K$ is an ordered valued differential field with small derivation and convex valuation ring.
Let $\bm k\langle y \rangle$ be an \emph{ordered} differential field extension of $\bm k$ with $y$ $\d$-algebraic over $\bm k$.
Then there exists an ordered valued differential field extension $K\langle a \rangle$ of $K$ such that:
\begin{enumerate}
    \item $\Gamma_{K \langle a \rangle} = \Gamma$;
    \item $K \langle a \rangle$ has small derivation and convex valuation ring;
    \item $a\asymp1$ and $\res(K\langle a\rangle)=\bm k\langle \overline{a} \rangle \cong \bm k\langle y \rangle$ over $\bm k$ \textnormal{(}as ordered differential fields\textnormal{)};
    \item for any ordered valued differential field extension $M$ of $K$ with convex valuation ring that is $\d$-henselian, every embedding $\bm k\langle \overline{a}\rangle \to \bm k_M$ over $\bm k$ is induced by an embedding $K\langle a\rangle \to M$ over~$K$.
\end{enumerate}
Moreover, if $K$ is a pre-$H$-field with gap~$0$, then so is $K\langle a \rangle$.
\end{lem}
\begin{proof}
The existence of the valued differential field $K\langle a\rangle$ and its embedding property as a valued differential field are provided by \cite[Theorem~6.3.2 and Lemma~7.1.4]{adamtt}.
Equipping $K\langle a\rangle$ with the ordering from Lemma~\ref{lem:ordresfldext} gives the rest, with the embedding property as an \emph{ordered} valued differential field following from the uniqueness of the ordering on $K\langle a\rangle$.
\end{proof}

For the $\d$-transcendental analogue, which likewise follows from \cite[Lemma~6.3.1]{adamtt} and Lemma~\ref{lem:ordresfldext}, we construe the fraction field $K \langle Y \rangle$ of $K\{Y\}$ as a valued differential field extension of $K$ by extending $\der$ and the map $P \mapsto v(P)$ to $K \langle Y \rangle$ in the obvious way.
Then $\Gamma_{K \langle Y \rangle}=\Gamma$, and $K\langle Y\rangle$ has small derivation if $K$ does, in which case $\overline{Y}$ is $\d$-transcendental over $\bm k$ and $\res(K\langle Y\rangle) = \bm k\langle \overline{Y}\rangle$ (see \cite[\S6.3]{adamtt}).
\begin{lem}\label{adh6.3.1o}
Suppose that $K$ is an ordered valued differential field with small derivation and convex valuation ring.
Suppose that $\bm k\langle \overline{Y} \rangle$ is an \emph{ordered} differential field extension of $\bm k$ and equip $K \langle Y \rangle$ with the ordering from Lemma~\ref{lem:ordresfldext}.
Let $M$ be an ordered valued differential field extension of $K$ with small derivation, convex valuation ring, and $a \in M$ with $a \asymp 1$ and $\overline{a}$ $\d$-transcendental over $\bm k$.
Then there exists a unique embedding $K \langle Y \rangle \to M$ over $K$ with $Y \mapsto a$.
If $K$ is a pre-$H$-field with gap~$0$, then so is $K\langle Y \rangle$.
\end{lem}

Similarly, we have an ordered variant of \cite[Corollary~7.1.5]{adamtt}.
\begin{cor}\label{adh6.3.3o}
Suppose that $K$ is an ordered valued differential field with small derivation and convex valuation ring.
Let $\bm k_L$ be an ordered differential field extension of $\bm k$.
Then $K$ has an ordered valued differential field extension $L$ with the following properties:
\begin{enumerate}
    \item $\Gamma_L = \Gamma$;
    \item $L$ has small derivation and convex valuation ring;
    \item $\res(L) \cong \bm k_L$ over $\bm k$ \textnormal{(}as ordered differential fields\textnormal{)};
    \item for any ordered valued differential field extension $M$ of $K$ with convex valuation ring that is $\d$-henselian, every embedding $\res(L) \to \bm k_M$ over $\bm k$ is induced by an embedding $L \to M$ over~$K$.
\end{enumerate}
Moreover, if $K$ is a pre-$H$-field with gap~$0$, then so is $L$.
\end{cor}

\section{Extensions controlled by the asymptotic couple}\label{sec:extac}
Towards our quantifier elimination and model completion results for pre-$H$-fields with gap~$0$, in this section we study extensions controlled by the asymptotic couple.
We begin by studying $H$-asymptotic couples with gap~$0$ in their own right and finding the model completion of this theory, which has quantifier elimination.

\subsection{Asymptotic couples with small derivation}\label{sec:ac-small}
The material in this subsection is based on \cite{adh-revisitingclosedac}, which, apart from its new results, revisits quantifier elimination for the theory of closed $H$-asymptotic couples from \cite{ad-ac}, introducing several new lemmas that simplify the arguments.
For convenience, we use \cite[\S6.5 and \S9.2]{adamtt} as a reference instead of the original sources \cite{rosen-dval,rosen-dvalgp1,rosen-dvalgp2,ad-ac,ad-hf} for asymptotic couples.

In contrast with the results of \cite{ad-ac,adh-revisitingclosedac}, here we do not need to expand the language by a predicate for the $\Psi$-set or by functions for divisibility by nonzero natural numbers.
Additionally, those authors work over an arbitrary ordered scalar field $\bm k$, but here we work over $\bb Q$ for concreteness (the results of this section hold in that setting in the language $\ca L_{\ac}$, described later, expanded by functions for scalar multiplication).

\subsubsection{Preliminaries}
We suspend in this subsection the convention that $\Gamma$ is the value group of $K$.
Instead, $(\Gamma, \psi)$ is an \deft{$H$-asymptotic couple}, which means that $\Gamma$ is an ordered abelian group and $\psi \colon \Gamma^{\neq} \to \Gamma$ is a map satisfying, for all $\gamma, \delta \in \Gamma^{\neq}$:
\begin{enumerate}[label=(AC\arabic*)]
    \item\label{ac1} if $\gamma+\delta \neq 0$, then $\psi(\gamma+\delta) \ges \min\{\psi(\gamma),\psi(\delta)\}$;
    \item\label{ac2} $\psi(k\gamma) = \psi(\gamma)$ for all $k \in \Z^{\neq}$;
    \item\label{ac3} if $\gamma>0$, then $\gamma+\psi(\gamma)>\psi(\delta)$;
    \myitem[(HC)]\manuallabel{hc}{(HC)} if $0<\gamma \les \delta$, then $\psi(\gamma) \ges \psi(\delta)$.
\end{enumerate}
Keeping in mind that later $(\Gamma, \psi)$ will be the asymptotic couple of an $H$-asymptotic field (such as a pre-$H$-field), we let $\gamma^\dagger \coloneqq \psi(\gamma)$ and $\gamma' \coloneqq \gamma^\dagger + \gamma$ for $\gamma \in \Gamma^{\neq}$.
The map $\gamma \mapsto \gamma'$ for $\gamma \in \Gamma^{\neq}$ is strictly increasing \cite[Lemma~6.5.4(iii)]{adamtt}.
We let $\Psi \coloneqq \psi(\Gamma^{\neq})$ and for any ordered abelian group $G$ we set $G^{\les} \coloneqq \{ g \in G : g \les 0 \}$, having earlier defined $G^>$ and $G^<$.
Thus \ref{ac3} says that $\Psi<(\Gamma^>)'$.

A fundamental result is that $\Gamma \setminus (\Gamma^{\neq})'$ has at most one element \cite[Theorem~9.2.1]{adamtt}, and moreover exactly one of the following holds \cite[Corollary~9.2.16]{adamtt}:
\begin{enumerate}
    \item there exists $\beta \in \Gamma$ such that $\Psi<\beta<(\Gamma^>)'$, in which case $(\Gamma, \psi)$ has \deft{gap~$\beta$};
    \item there exists $\beta \in \Gamma$ such that $\max\Psi=\beta$, in which case $(\Gamma, \psi)$ has \deft{max~$\beta$};
    \item $(\Gamma^{\neq})'=\Gamma$, in which case $(\Gamma, \psi)$ has \deft{asymptotic integration}.
\end{enumerate}
We are primarily concerned with $H$-asymptotic couples having gap~$0$, although we also consider the case of max~$0$ in this subsection.
The material on $H$-asymptotic couples with max~$0$ is only used in one later theorem that itself is not used in the main results, but fits naturally with that of the gap~$0$ case.
In contrast, the main results of \cite{ad-ac,adh-revisitingclosedac} concern asymptotic couples with asymptotic integration, such as the asymptotic couple of $\T$.
In adapting their arguments, we try to highlight how the substitution of ``gap~$0$'' or ``max~$0$'' for ``asymptotic integration'' (or similar changes) alters the proofs.

Note that if $(\Gamma, \psi)$ is the asymptotic couple of an asymptotic field $K$, then $K$ has gap~$0$ (in the sense of \S\ref{sec:prelim}) if and only if $(\Gamma, \psi)$ has gap~$0$.
It follows from \cite[Theorem~9.2.1 and Corollary~9.2.4]{adamtt} that $\sup\Psi=0 \iff (\Gamma^>)' = \Gamma^>$ and that $\sup\Psi=0 \notin \Psi$ if and only if $(\Gamma, \psi)$ has gap~$0$.
Thus $\sup\Psi=0$ if and only if $(\Gamma, \psi)$ has gap~$0$ or max~$0$.

It follows from \ref{ac2} and \ref{hc} that $\psi$ is constant on archimedean classes of $\Gamma$.
For $\gamma \in \Gamma$, we let $[\gamma] \coloneqq \{ \delta \in \Gamma : |\delta| \les n|\gamma|\ \text{and}\ |\gamma| \les n|\delta|\ \text{for some}\ n\}$ denote its archimedean class, and set $[\Gamma] \coloneqq \{ [\gamma] : \gamma \in \Gamma \}$, ordering it by $[\delta]<[\gamma]$ if $n|\delta|<|\gamma|$ for all $n$, where $\gamma, \delta \in \Gamma$.
The map $\psi$ extends uniquely to the divisible hull $\Q\Gamma$ of $\Gamma$, defined by $\psi(q\gamma)=\psi(\gamma)$ for $\gamma \in \Gamma^{\neq}$ and $q \in \Q^{\x}$ (use \cite[Lemma~6.5.3]{adamtt} to get \ref{ac3}), and in this way we always construe $\Q\Gamma$ as an $H$-asymptotic couple $(\Q\Gamma, \psi)$ extending $(\Gamma, \psi)$.
It satisfies $\psi((\Q\Gamma)^{\neq}) = \psi(\Gamma^{\neq})$, so if $(\Gamma, \psi)$ has gap~$0$ (respectively, max~$0$), then so does $(\Q\Gamma, \psi)$.
We also use that if $(\Gamma, \psi)$ has gap~$0$, then for every $\gamma \in \Gamma^{\neq}$, we have $[\gamma]>[\gamma^\dagger]$ by \cite[Lemma~9.2.10(iv)]{adamtt} and so $[\gamma]=[\gamma']$; in particular, if $\gamma<0$, then $\gamma<\gamma^\dagger<0$.

We call $(\Gamma, \psi)$ \deft{gap-closed} if $\Gamma$ is nontrivial and divisible, and $\Psi=\Gamma^<$.
The goal of this section is Theorem~\ref{ac:qe}, which states that the theory of gap-closed $H$-asymptotic couples has quantifier elimination and is the model completion of the theory of $H$-asymptotic couples with gap~$0$.
Likewise, the theory of max-closed $H$-asymptotic couples, where $(\Gamma, \psi)$ is \deft{max-closed} if $\Gamma$ is divisible and $\Psi=\Gamma^{\les}$, has quantifier elimination and is the model completion of the theory of $H$-asymptotic couples $(\Gamma, \psi)$ with $\sup\Psi=0$.
The language for these results is the language $\ca L_{\ac}=\{+, -, \les, 0, \infty, \psi\}$ of asymptotic couples.
The underlying set of $(\Gamma, \psi)$ in this language is $\Gamma_{\infty} \coloneqq \Gamma \cup \{\infty\}$, and we interpret $\infty$ in the following way: for all $\gamma \in \Gamma$, $\infty+\gamma=\gamma+\infty \coloneqq \infty$ and $\gamma<\infty$; $\infty+\infty \coloneqq \infty$; $-\infty \coloneqq \infty$; $\psi(0)=\psi(\infty) \coloneqq \infty$.
The other symbols have the expected interpretations.

One of the new extension lemmas of \cite{adh-revisitingclosedac} is \cite[Lemma~2.7]{adh-revisitingclosedac}.
Here is a variant that avoids assuming asymptotic integration by making a stronger cofinality assumption.
\begin{lem}\label{bm:2.7}
Suppose that $\Psi$ is downward closed in $\Gamma$. Let $(\Gamma_1, \psi_1)$ and $(\Gamma_*, \psi_*)$ be $H$-asymptotic couples extending $(\Gamma, \psi)$ such that $\Gamma^<$ is cofinal in $\Gamma_1^<$.
Suppose that $\gamma_1 \in \Gamma_1 \setminus \Gamma$ and $\gamma_* \in \Gamma_* \setminus \Gamma$ realize the same cut in $\Gamma$ and $\gamma_1^\dagger \notin \Gamma$.
Then $\gamma_*^\dagger \notin \Gamma$ and $\gamma_*^\dagger$ realizes the same cut in $\Gamma$ as~$\gamma_1^\dagger$.
\end{lem}
\begin{proof}
Let $\alpha \in \Gamma^{\neq}$.
We first show that:
\[
\gamma_1^\dagger < \alpha^\dagger \implies \gamma_*^\dagger < \alpha^\dagger \qquad \text{and} \qquad
\gamma_1^\dagger > \alpha^\dagger \implies \gamma_*^\dagger > \alpha^\dagger.
\]

The first implication follows exactly as in \cite[Lemma~2.7]{adh-revisitingclosedac}, which uses that $\Psi$ is downward closed but not that $(\Gamma, \psi)$ has asymptotic integration.
The second implication has a hidden use of asymptotic integration, so here are the modifications.

Suppose that $\gamma_1^\dagger>\alpha^\dagger$, so by the cofinality assumption, there is $\delta \in \Gamma$ with $\gamma_1^\dagger > \delta > \alpha^\dagger$.
By the reasoning in the first implication (see \cite[Lemma~2.7]{adh-revisitingclosedac}), it suffices to show that $\delta \in \Psi$; asymptotic integration is used in \cite[Lemma~2.7]{adh-revisitingclosedac} to get this.
By the stronger cofinality assumption here, take $\beta \in \Gamma^{\neq}$ with $|\gamma_1|>|\beta|$, so $\delta<\gamma_1^\dagger<\beta^\dagger \in \Psi$.
Thus $\delta \in \Psi$, since $\Psi$ is downward closed.

Finally, the last paragraph also shows that $\gamma_*^{\dagger} \notin \Gamma$, since $\gamma_*^{\dagger}\les\beta^\dagger \in \Psi$ and $\Psi$ is downward closed, and therefore it realizes the same cut as $\gamma_1^{\dagger}$ by the displayed implications.
\end{proof}
Alternatively, we can weaken the cofinality assumption from $\Gamma^<$ being cofinal in $\Gamma_1^<$ to $\Gamma^<$ being cofinal in $(\Gamma + \Q\gamma_1^\dagger)^<$ (as in \cite[Lemma~2.7]{adh-revisitingclosedac}) in the case that $\sup\Psi_1=0$; the only change in the proof is that now $\delta \in \Psi$ follows from $\delta<\gamma_1^\dagger \les 0$.

Another of the new lemmas is the following, \cite[Lemma~2.8]{adh-revisitingclosedac}, which we also use in~\S\ref{sec:16.1}.
\begin{lem}\label{bm:2.8}
Suppose that $(\Gamma_1, \psi_1)$ is an $H$-asymptotic couple extending $(\Gamma, \psi)$, and let $\gamma_1 \in \Gamma_1 \setminus \Gamma$ and $\alpha \in \Gamma$.
Suppose that $\gamma_1$ and $\beta \coloneqq \gamma_1^\dagger-\alpha$ satisfy $\gamma_1^\dagger \notin \Gamma$ and $\beta^\dagger \notin \Psi$, and that $|\gamma_1|>|\gamma|$ for some $\gamma \in \Gamma^{\neq}$.
Then $\gamma_1^\dagger < \beta^\dagger$.
\end{lem}

\subsubsection{Embedding lemmas}
We now turn to the proof of quantifier elimination for gap-closed and max-closed $H$-asymptotic couples, which we handle uniformly as much as possible.
To that end, suppose that $(\Gamma, \psi)$ is a divisible $H$-asymptotic couple with $\sup\Psi=0$, and let $(\Gamma_1, \psi_1)$ and $(\Gamma_*, \psi_*)$ be divisible $H$-asymptotic couples extending $(\Gamma, \psi)$ such that $(\Gamma_*, \psi_*)$ is $|\Gamma|^+$-saturated.
Let $\gamma_1 \in \Gamma_1 \setminus \Gamma$ and $(\Gamma \langle \gamma_1 \rangle, \psi_1)$ be the divisible $H$-asymptotic couple generated by $\Gamma\cup\{\gamma_1\}$ in $(\Gamma_1, \psi_1)$.

The first two lemmas are proved similarly to \cite[Lemmas~3.4 and 3.5]{adh-revisitingclosedac}, so we only sketch the proofs, omitting some details that are the same.
The third lemma is particular to the case of gap~$0$. 
For convenience, we set $0^\dagger \coloneqq \psi(0) = \infty$, so $\Gamma^\dagger = \Psi \cup \{\infty\}$.

In the next lemma, the $\sup\Psi=0$ assumption simplifies the case distinctions needed in \cite[Lemma~3.4]{adh-revisitingclosedac}, and the assumption of asymptotic integration there is used only in \textit{Cases 1--2} of that lemma, which do not occur here.
We also do not need the ``$H$-cuts'' of that paper for quantifier elimination since the $\Psi$-set in a gap-closed or a max-closed $H$-asymptotic couple is already quantifier-free definable without parameters.

\begin{lem}\label{nonewdaggers}
Suppose that $(\Gamma + \bb Q\gamma_1)^\dagger = \Gamma^\dagger$.
Then $(\Gamma\langle \gamma_1 \rangle, \psi_1)$ can be embedded into $(\Gamma_*, \psi_*)$ over~$\Gamma$.
\end{lem}
\begin{proof}
From $(\Gamma+\bb Q\gamma_1)^\dagger = \Gamma^\dagger$, we get $\Gamma\langle \gamma_1 \rangle = \Gamma+\bb Q\gamma_1$.

In case $[\Gamma+\bb Q\gamma_1] = [\Gamma]$, realizing in $\Gamma_*$ the cut in $\Gamma$ realized by $\gamma_1$ yields an embedding $i \colon \Gamma+\bb Q\gamma_1 \to \Gamma_*$ of ordered $\Q$-vector spaces fixing $\Gamma$, and the assumption $[\Gamma+\bb Q\gamma_1] = [\Gamma]$ ensures that it is an embedding of $H$-asymptotic couples (see \textit{Case 3} of \cite[Lemma~3.4]{adh-revisitingclosedac}).

In case $[\Gamma+\bb Q\gamma_1] \neq [\Gamma]$ but $\Gamma^<$ is cofinal in $(\Gamma+\bb Q\gamma_1)^<$, take $\beta \in \Gamma_1 \setminus \Gamma$ with $\beta>0$ and $[\beta] \notin [\Gamma]$.
Letting $D_1$ be the cut in $\Gamma$ realized by $\beta$ and $D_2 \coloneqq \Gamma \setminus D_1$, it follows from the cofinality assumption that $D_1$ has no greatest element and $D_2$ has no least element.
Then saturation gives $\beta_* \in \Gamma_*$ realizing the same cut as $\beta$ in $\Gamma$ and with $\beta^\dagger = \beta_*^\dagger$, yielding an embedding $i \colon \Gamma+\bb Q\gamma_1 \to \Gamma_*$ of $H$-asymptotic couples fixing $\Gamma$ (see \textit{Cases 4--6} of \cite[Lemma~3.4]{adh-revisitingclosedac}).

In case $\Gamma^<$ is not cofinal in $(\Gamma+\bb Q\gamma_1)^<$, take $\beta \in \Gamma+\bb Q\gamma_1$ satisfying $0<\beta<\Gamma^>$.
In this case, $(\Gamma, \psi)$ must have max~$0$, since if $(\Gamma, \psi)$ has gap~$0$, then $\Psi<\psi_1(\beta) \in \Psi$, a contradiction.
By saturation, take $\beta_* \in \Gamma_*$ with $0<\beta_*<\Gamma^>$ and $\beta_*^\dagger=\beta^\dagger=0$, so we get an embedding $i \colon \Gamma+\bb Q\gamma_1 \to \Gamma_*$ of $H$-asymptotic couples as before.
\end{proof}

In \cite[Lemma~3.5]{adh-revisitingclosedac}, the cofinality of $\Gamma^<$ in $\Gamma_1^<$ follows from asymptotic integration;
In the max~$0$ case it is also automatic, but needs to be added as an assumption in the gap~$0$ case, \ref{newdaggerscof:gap} in the next lemma.
\begin{lem}\label{newdaggerscof}
Suppose that $(\Gamma+\bb Q\gamma)^\dagger \neq \Gamma^\dagger$ for all $\gamma \in \Gamma_1 \setminus \Gamma$.
Also, suppose that either:
\begin{enumerate}
    \item\label{newdaggerscof:gap} $(\Gamma, \psi)$ is gap-closed, $(\Gamma_1, \psi_1)$ has gap~$0$, and $\Gamma^<$ is cofinal in $\Gamma_1^<$; or
    \item\label{newdaggerscof:max} $(\Gamma, \psi)$ is max-closed and $(\Gamma_1,\psi_1)$ has max~$0$.
\end{enumerate}
Then $(\Gamma\langle\gamma_1\rangle, \psi_1)$ can be embedded into $(\Gamma_*, \psi_*)$ over~$\Gamma$.
\end{lem}
\begin{proof}
Note that in case~\ref{newdaggerscof:max}, the cofinality assumption comes for free, for if $\gamma \in \Gamma_1 \setminus \Gamma$ satisfies $0<\gamma<\Gamma^>$, then $\gamma^\dagger = 0$ and so $(\Gamma+\bb Q\gamma)^\dagger = \Gamma^\dagger$, a contradiction.

Take $\alpha_1 \in \Gamma$ such that $(\gamma_1-\alpha_1)^\dagger \notin \Gamma^\dagger$.
Then $(\gamma_1-\alpha_1)^\dagger \notin \Gamma$, since $(\gamma_1-\alpha_1)^\dagger<0$ and in case~\ref{newdaggerscof:gap} and case~\ref{newdaggerscof:max}, $\Psi=\Gamma^{<}$ and $\Psi=\Gamma^{\les}$, respectively.
(In \cite[Lemma~3.5]{adh-revisitingclosedac}, this step uses asymptotic integration.)

The proof now proceeds exactly as in \cite[Lemma~3.5]{adh-revisitingclosedac}, substituting Lemma~\ref{bm:2.7} for \cite[Lemma~2.7]{adh-revisitingclosedac} (and $\Q$ for $\bm k$); here is a sketch.
Let $n \ges 1$.
Continue this procedure to construct sequences $\alpha_1, \alpha_2, \dots$ in $\Gamma$ and $\beta_1, \beta_2, \dots$ in $\Gamma\langle \gamma_1 \rangle \setminus \Gamma$ with $\beta_1 = \gamma_1-\alpha_1$ and $\beta_{n+1} = \beta_n^\dagger-\alpha_{n+1}$.
Then $\beta_n^\dagger<\beta_{n+1}^\dagger$ by Lemma~\ref{bm:2.8} (with $\beta_n$ in place of $\gamma_1$ and $\beta_{n+1}$ in place of $\beta$), so $[\beta_n]>[\beta_{n+1}]$, so
\[
\Gamma\langle \gamma_1 \rangle\ =\ \Gamma \oplus \bb Q\beta_1 \oplus \bb Q\beta_2 \oplus \cdots.
\]
Saturation gives $\gamma_* \in \Gamma_* \setminus \Gamma$ realizing the same cut in $\Gamma$ as $\gamma_1$, so use the same $\alpha_n$ sequence to likewise define $\beta_{*n} \in (\Gamma_*)_{\infty}$ for $n \ges 1$.
Inductively construct an increasing sequence of $\Q$-vector space embeddings
\[
i_n \colon \Gamma + \bb Q\beta_1 + \dots + \bb Q\beta_n \to \Gamma_*
\]
fixing $\Gamma$ and sending $\beta_n$ to $\beta_{*n}$ in two steps:
Given that $\beta_n$ and $\beta_{*n}$ realize the same cut in $\Gamma$, use Lemma~\ref{bm:2.7} to get that $\beta_n^{\dagger}$ and $\beta_{*n}^{\dagger}$ realize the same cut in $\Gamma$.
Then argue with archimedean classes, using Lemma~\ref{bm:2.8} in $(\Gamma_*,\psi_*)$ now, to extend the embedding to $i_{n+1}$.
The union of these maps is the desired embedding of $H$-asymptotic couples.
\end{proof}

In case~\ref{newdaggerscof:gap} of the previous lemma, we added a cofinality assumption that was not present in \cite[Lemma~3.5]{adh-revisitingclosedac}.
The next lemma handles the non-cofinal case, which is particular to gap~$0$.
\begin{lem}\label{newdaggersnocof}
Suppose that $\Gamma^< < \gamma_1 <0$, and $(\Gamma_1,\psi_1)$ and $(\Gamma_*,\psi_*)$ have gap~$0$.
Then $(\Gamma\langle\gamma_1\rangle, \psi_1)$ can be embedded into $(\Gamma_*, \psi_*)$ over~$\Gamma$.
\end{lem}
\begin{proof}
Set $\gamma_1^{\langle 0\rangle} \coloneqq \gamma_1$ and $\gamma_1^{\langle n+1 \rangle} \coloneqq (\gamma_1^{\langle n \rangle})^\dagger$ for all $n$.
We have
\[
[\gamma_1]>[\gamma_1^{\langle 1 \rangle}]>[\gamma_1^{\langle 2 \rangle}]>\cdots,
\]
and so
\[
\Gamma^< < \gamma_1 < \gamma_1^{\langle 1 \rangle} < \gamma_1^{\langle 2 \rangle}<\dots<0 \qquad \text{and} \qquad [\gamma_1^{\langle n \rangle}] \notin [\Gamma]\ \text{for all}\ n.
\]
Hence the family $(\gamma_1^{\langle n \rangle})_{n \in \bb N}$ is $\bb Q$-linearly independent over $\Gamma$ and
\[
\Gamma\langle \gamma_1 \rangle\ =\ \Gamma \oplus \bb Q\gamma_1 \oplus \bb Q\gamma_1^{\langle 1 \rangle} \oplus \bb Q\gamma_1^{\langle 2 \rangle} \oplus \cdots.
\]
By saturation, we may take $\gamma_* \in \Gamma_* \setminus \Gamma$ with $\Gamma^< < \gamma_* < 0$.
The above holds in $\Gamma_*$ with $\gamma_*$ replacing $\gamma_1$ (and $\gamma_*^{\langle n \rangle}$ defined analogously), so we obtain an embedding of $(\Gamma\langle \gamma_1 \rangle, \psi_1)$ into $(\Gamma_*, \psi_*)$ over $\Gamma$ that sends $\gamma_1$ to $\gamma_*$.
\end{proof}

\subsubsection{Quantifier elimination}

We call an $H$-asymptotic couple $(\Gamma_1, \psi_1)$ extending $(\Gamma, \psi)$ a \deft{gap-closure} of $(\Gamma, \psi)$ if it is gap-closed and it embeds over $(\Gamma, \psi)$ into every gap-closed $H$-asymptotic couple extending $(\Gamma, \psi)$.
Similarly, we call an $H$-asymptotic couple $(\Gamma_1, \psi_1)$ extending $(\Gamma, \psi)$ a \deft{max-closure} of $(\Gamma, \psi)$ if it is max-closed and it embeds over $(\Gamma, \psi)$ into every max-closed $H$-asymptotic couple extending $(\Gamma, \psi)$.
By the embedding lemmas of the previous subsection and a standard quantifier elimination test (see for example \cite[Corollary~B.11.11]{adamtt}), quantifier elimination for gap-closed and max-closed $H$-asymptotic couples reduces to the existence of gap-closures and max-closures, respectively.

For that, we need one more embedding lemma,
Let $\Psi^{\downarrow}$ be the downward closure of $\Psi$ in~$\Gamma$.
\begin{lem}\label{bm:3.1}
Let $\beta \in \Psi^{\downarrow} \setminus \Psi$ or $\beta$ be a gap in $(\Gamma, \psi)$.
Then there is an $H$-asymptotic couple $(\Gamma \oplus \bb Z\alpha, \psi^\alpha)$ extending $(\Gamma, \psi)$ such that:
\begin{enumerate}
	\item $\alpha>0$ and $\psi^\alpha(\alpha)=\beta$;
	\item\label{bm:3.1ii} given an embedding $i$ of $(\Gamma, \psi)$ into an $H$-asymptotic couple $(\Gamma^*, \psi^*)$ and $\alpha^* \in \Gamma^*$ with $\alpha^*>0$ and $\psi^*(\alpha^*)=i(\beta)$, there is a unique extension of $i$ to an embedding $j \colon (\Gamma \oplus \bb Z\alpha, \psi^\alpha) \to (\Gamma^*, \psi^*)$ with $j(\alpha)=\alpha^*$.
\end{enumerate}
\end{lem}
\begin{proof}
This follows from \cite[Lemma~9.8.7]{adamtt} with $C=\{[\gamma] : \gamma \in \Gamma^{\neq},\ \psi(\gamma)>\beta\}$, but here is an outline.
Define the ordered abelian group $\Gamma \oplus \bb Z\alpha$ so that
\[
\{0\} \cup \{ \gamma \in \Gamma^> : \psi(\gamma)>\beta \}\ <\ \alpha\ <\ \{ \gamma \in \Gamma^> : \psi(\gamma)<\beta \},
\]
and thus $[\alpha] \notin [\Gamma^{\neq}]$.
Extend $\psi$ to $\psi^{\alpha} \colon (\Gamma \oplus \bb Z\alpha)^{\neq} \to \Gamma$ by $\psi^{\alpha}(\gamma+k\alpha)\coloneqq \min\{\psi(\gamma),\beta\}$, where $\gamma \in \Gamma$ and $k \in \bb Z^{\neq}$.
To verify that this makes $(\Gamma \oplus \bb Z\alpha, \psi^\alpha)$ an $H$-asymptotic couple involves tedious case distinctions.
The axioms \ref{ac1}, \ref{ac2}, and \ref{hc} are straightforward, while \ref{ac3} is more subtle and uses \cite[Lemma~6.5.4]{adamtt}.
The universal property of $(\Gamma \oplus \bb Z\alpha, \psi^\alpha)$ is easy.
\end{proof}

\begin{cor}\label{hclosure}
Every $H$-asymptotic couple $(\Gamma, \psi)$ with gap~$0$ has a gap-closure.
Every $H$-asymptotic couple $(\Gamma, \psi)$ with $\sup\Psi=0$ has a max-closure.
\end{cor}
\begin{proof}
If $(\Gamma,\psi)$ has max~$0$, then by applying Lemma~\ref{bm:3.1} for $\beta \in \Psi^{\downarrow} \setminus \Psi$, taking the divisible hull, and repeating these procedures, we get a max-closure of $(\Gamma, \psi)$.
If $(\Gamma,\psi)$ has gap~$0$ and $\Gamma\neq\{0\}$, the same procedure yields a gap-closure of $(\Gamma, \psi)$.
To obtain a max-closure of $(\Gamma,\psi)$ when it has gap~$0$, first apply Lemma~\ref{bm:3.1} with $\beta=0$ and then repeat the previous argument.

It remains to show that the trivial $H$-asymptotic couple $(\{0\},\psi)$, where $\psi \colon \0 \to \{0\}$ is the empty function, has a gap-closure.
Let $\Gamma_0 \coloneqq \bigoplus_{n} \Q\gamma_n$ be the ordered $\Q$-vector space
satisfying $\gamma_n<0$ and $[\gamma_n]>[\gamma_{n+1}]$ for all $n$, and equip it with the function $\psi_0$ defined by $\psi_0(q_1\gamma_{n_1}+\dots+q_m\gamma_{n_m})=\gamma_{n_1+1}$ for all $m \ges 1$, $n_1>\dots >n_m$, and $q_1, \dots, q_m \in \Q^{\x}$.
Then by the proof of Lemma~\ref{newdaggersnocof}, $(\Gamma_0, \psi_0)$ is an $H$-asymptotic couple with gap~$0$ that embeds into every nontrivial $H$-asymptotic couple with gap~$0$.
Hence, a gap-closure of $(\Gamma_0,\psi_0)$ is a gap-closure of $(\{0\},\psi)$.
\end{proof}

Next we establish quantifier elimination, first proving it in $\ca L_{\ac, \divop} \coloneqq \ca L_{\ac} \cup \{\divop_n : n \ges 1\}$, the language $\ca L_{\ac}$ expanded by unary function symbols $\divop_n$ interpreted as division by $n$ with $\divop_n(\infty) \coloneqq \infty$.
As noted already, unlike in \cite{ad-ac,adh-revisitingclosedac}, we do not need the predicate for the $\Psi$-set, since it is already quantifier-free definable without parameters.
\begin{thm}\label{ac:qe}
The theory of gap-closed $H$-asymptotic couples has quantifier elimination, and it is the model completion of the theory of $H$-asymptotic couples with gap~$0$.
The theory of max-closed $H$-asymptotic couples has quantifier elimination, and it is the model completion of the theory of $H$-asymptotic couples $(\Gamma, \psi)$ with $\sup\Psi=0$.
\end{thm}
\begin{proof}
That these theories have quantifier elimination in $\ca L_{\ac, \divop}$ follows from Lemmas~\ref{nonewdaggers}, \ref{newdaggerscof}, and \ref{newdaggersnocof}, and Corollary~\ref{hclosure} by a standard quantifier elimination test.
To see that they have quantifier elimination in $\ca L_{\ac}$, recall how, for an $H$-asymptotic couple $(\Gamma, \psi)$, $\psi$ extends uniquely to the divisible hull $\Q\Gamma$ of $\Gamma$, preserving the property of having gap~$0$ or max~$0$.

The model completion statements follow from quantifier elimination and Corollary~\ref{hclosure}.
\end{proof}

\begin{cor}\label{ac:complete}
The theory of gap-closed $H$-asymptotic couples is complete and has a prime model.
The theory of max-closed $H$-asymptotic couples is complete and has a prime model.
\end{cor}
\begin{proof}
The trivial $H$-asymptotic couple (see the proof of Corollary~\ref{hclosure}), embeds into every $H$-asymptotic couple, yielding completeness of the two theories.
It also has gap~$0$, so its gap-closure and its max-closure are the respective prime models.
\end{proof}
As prime models of their respective complete theories, on standard model-theoretic grounds (see for instance \cite[Corollary~4.2.16]{marker}) the gap-closure and the max-closure of the trivial $H$-asymptotic couple are unique up to isomorphism.
By quantifier elimination and adding constants to the language, similar reasoning applies to the gap-closure and the max-closure of a countable $H$-asymptotic couple.
However, we do not know in general if gap-closures and max-closures are unique up to isomorphism.
One possibility is foreclosed:
By adapting the arguments of \cite[\S3]{a-remac}, we can show that if an $H$-asymptotic couple $(\Gamma, \psi)$ with gap~$0$ is not gap-closed, then a gap-closure of $(\Gamma, \psi)$ is not minimal in the sense that it has a proper gap-closed substructure containing $\Gamma$.
The same is true replacing ``gap~$0$'' with ``$\sup\Psi=0$'', ``gap-closed'' with ``max-closed'', and ``gap-closure'' with ``max-closure''.
This nonminimality result has no bearing on the rest of the paper, so we do not present details here.
With more work, one could likely obtain further model-theoretic results on such $H$-asymptotic couples, along the lines of \cite{ad-ac,adh-revisitingclosedac}.

We do use Theorem~\ref{ac:qe} via the following corollary.
For $n \ges 1$, $\alpha_1, \dots, \alpha_n \in \Gamma$, and $\gamma \in \Gamma$, we define the function $\psi_{\alpha_1,\dots,\alpha_n} \colon \Gamma_\infty \to \Gamma_\infty$ recursively by
\[
\psi_{\alpha_1}(\gamma) \coloneqq \psi(\gamma-\alpha_1) \qquad \text{and} \qquad \psi_{\alpha_1,\dots,\alpha_{n}}(\gamma) \coloneqq \psi\big(\psi_{\alpha_1,\dots,\alpha_{n-1}}(\gamma)-\alpha_{n}\big)\ \text{for}\ n \ges 2.
\]

\begin{cor}\label{ac:9.9.2}
Let $(\Gamma, \psi)$ be a gap-closed $H$-asymptotic couple 
and let $(\Gamma^*, \psi^*)$ be an $H$-asymptotic couple extending $(\Gamma, \psi)$ with gap~$0$.
Suppose $n \ges 1$, $\alpha_1, \dots, \alpha_n \in \Gamma$, $q_1, \dots, q_n \in \Q$, and $\gamma^* \in \Gamma^*$ are such that:
\begin{enumerate}
    \item $\psi^*_{\alpha_1, \dots, \alpha_n}(\gamma^*) \neq \infty$ \textnormal{(}so $\psi^*_{\alpha_1, \dots, \alpha_i}(\gamma^*) \neq \infty$ for $i=1,\dots,n$\textnormal{)};
    \item $\gamma^* + q_1 \psi^*_{\alpha_1}(\gamma^*) + \dots + q_n \psi^*_{\alpha_1, \dots, \alpha_n}(\gamma^*) \in \Gamma$ \textnormal{(}in $\Q\Gamma^*$\textnormal{)}.
\end{enumerate}
Then $\gamma^* \in \Gamma$.
\end{cor}
\begin{proof}
By Theorem~\ref{ac:qe}, $(\Gamma, \psi)$ is an existentially closed $H$-asymptotic couple with gap~$0$, so we have $\gamma \in \Gamma$ with
\[
\gamma + q_1 \psi_{\alpha_1}(\gamma) + \dots + q_n \psi_{\alpha_1, \dots, \alpha_n}(\gamma)\
=\
\gamma^* + q_1 \psi^*_{\alpha_1}(\gamma^*) + \dots + q_n \psi^*_{\alpha_1, \dots, \alpha_n}(\gamma^*).
\]
It remains to use \cite[Lemma~9.9.3]{adamtt} to obtain $\gamma^*=\gamma \in \Gamma$.
\end{proof}

The same holds with ``max-closed'' replacing ``gap-closed''and ``$\sup\Psi=0$'' replacing ``gap~$0$''.

\subsection{A maximality theorem}\label{sec:16.1}
The results and proofs of this section are adapted from \cite[\S16.1]{adamtt}.
In the next lemma and its consequences, we use the quantifier elimination for gap-closed asymptotic couples from \S\ref{sec:ac-small} to study extensions of certain asymptotic fields whose asymptotic couples are gap-closed.
Note that if $K$ is an $H$-asymptotic field with exponential integration and gap~$0$, then $\Psi = \Gamma^<$, so if additionally $\Gamma$ is divisible then $(\Gamma, \psi)$ is gap-closed.

\begin{ass} In this subsection, $K$ has small derivation.\end{ass}
The next lemma follows from Corollary~\ref{ac:9.9.2} in the same way that \cite[Lemma~16.1.1]{adamtt} follows from \cite[Proposition~9.9.2]{adamtt}, except with ``$\d$-algebraically maximal'' replacing ``asymptotically $\d$-algebraically maximal'':
$K$ is \deft{differential-algebraically maximal} (\deft{$\d$-algebraically maximal} for short) if it has no proper differentially algebraic (``$\d$-algebraic'' for short) immediate extension with small derivation.
And instead of \cite[Theorem~14.0.2]{adamtt} (used implicitly in \cite[Lemma~16.1.1]{adamtt}), we need \cite[Theorem~3.6]{pc-dh}.
We also need the notion of a pseudocauchy sequence (``pc-sequence'' for short); for the definition and basic facts about pc-sequences, see \cite[\S2.2]{adamtt}, and for pc-sequences in valued differential fields, see \cite[\S4.4 and \S6.9]{adamtt}.
In an immediate extension $L$ of $K$, every element of $L \setminus K$ is the pseudolimit of a pc-sequence in $K$ that has no pseudolimit in $K$, called \deft{divergent in $K$}, and divergent pc-sequences in $K$ can be of $\d$-algebraic or $\d$-transcendental type over~$K$.
The exact definitions are not really important here, for which see the end of \cite[\S4.4]{adamtt}.
But note that if the derivation on $\bm k$ is nontrivial, then $K$ is $\d$-algebraically maximal if and only if there is no divergent pc-sequence in $K$ of $\d$-algebraic type over $K$ by \cite[Lemma~6.9.3]{adamtt}.

\begin{lem}\label{adh16.1.1}
Suppose that $K$ is a $\d$-henselian $H$-asymptotic field with exponential integration and gap~$0$ whose value group is divisible.
Let $L$ be an $H$-asymptotic extension of $K$ with gap~$0$ and $\bm k_L = \bm k$, and suppose that there is no $y \in L \setminus K$ such that $K\langle y \rangle$ is an immediate extension of $K$.
Let $f \in L \setminus K$.
Then the vector space $\Q\Gamma_{K\langle f \rangle}/\Gamma$ is infinite dimensional.
\end{lem}
\begin{proof}
First, we argue that there is no divergent pc-sequence in $K$ with a pseudolimit in $L$.
Towards a contradiction, suppose that $(a_\rho)$ is a divergent pc-sequence in $K$ with pseudolimit $\ell \in L$.
Since $K$ is $\d$-henselian and asymptotic, it is $\d$-algebraically maximal \cite[Theorem~3.6]{pc-dh}, so $(a_\rho)$ is not of $\d$-algebraic type over $K$.
Hence $(a_\rho)$ is of $\d$-transcendental type over $K$, so $K\langle \ell \rangle$ is an immediate extension of $K$ by \cite[Lemma~6.9.1]{adamtt}, a contradiction.

Thus for all $y \in L \setminus K$, the set $v(y-K) \subseteq \Gamma_L$ has a maximum.
Take $b_0 \in K$ with $v(f-b_0) = \max v(f-K)$, so $v(f-b_0) \notin \Gamma$ since $\bm k_L = \bm k$.
Set $f_1 \coloneqq (f-b_0)^{\dagger} \in L$.
Then $f_1 \notin K$, since otherwise, there would be $g \in K^{\x}$ with $(f-b_0)^\dagger = g^\dagger$, so $v(f-b_0) = vg \in \Gamma$, a contradiction.
Setting $f_0 \coloneqq f$, this leads to sequences $(f_n)$ in $L \setminus K$ and $(b_n)$ in $K$ such that for all~$n$:
\begin{enumerate}
    \item $v(f_n-b_n) = \max v(f_n-K) \notin \Gamma$;
    \item $f_{n+1} = (f_n-b_n)^\dagger$.
\end{enumerate}
To finish the proof of the lemma, show that
\[v(f_0-b_0),\ v(f_1-b_1),\ \dots\ \text{are $\Q$-linearly independent over}\ \Gamma.\]
If $\Gamma = \{ 0 \}$, then this follows from \cite[Lemma~9.2.10(iv)]{adamtt}.
Otherwise, the idea is that by taking for $n \ges 1$ an $a_n \in K^{\x}$ with $a_n^{\dagger}=b_n$ and setting $\alpha_n \coloneqq v(a_n)$, Corollary~\ref{ac:9.9.2} (with one of the $v(f_n-b_n)$ in the role of $\gamma$) shows that a $\Q$-linear dependence is impossible since each $v(f_n-b_n) \notin \Gamma$.
\end{proof}

The next result is a strengthening of \cite[Theorem~3.6]{pc-dh} under additional hypotheses, including that the asymptotic couple of $K$ is gap-closed.
It follows from the Zariski--Abhyankar inequality by combining \cite[Theorem~3.6]{pc-dh} and Lemma~\ref{adh16.1.1} just as \cite[Theorem~16.0.3]{adamtt} combines \cite[Theorem~14.0.2]{adamtt} and \cite[Lemma~16.1.1]{adamtt}.
This result is used in \S\ref{sec:dhl} to prove the minimality of $\d$-Hensel-Liouville closures.
\begin{thm}\label{dalgmaxupgrade}
Suppose that $K$ is a $\d$-henselian $H$-asymptotic field with exponential integration and gap~$0$ whose value group is divisible.
Then $K$ has no proper $\d$-algebraic $H$-asymptotic extension with gap~$0$ and the same residue field.
\end{thm}

By quantifier elimination for max-closed $H$-asymptotic couples, the above result holds with ``max~$0$'' replacing ``gap~$0$'', where we say an asymptotic field $K$ has \deft{max~$0$} if its asymptotic couple does.

We now summarize some facts about the asymptotic couple of the $K\langle f \rangle$ from Lemma~\ref{adh16.1.1} that we need in the next subsection.

\begin{lem}\label{adh16.1.2}
Let $K$, $L$, and $f$ be as in Lemma~\ref{adh16.1.1}, and let the sequences $(f_n)$, $(b_n)$, $(a_n)_{n \ges 1}$, and $(\alpha_n)_{n \ges 1}$ be as in the proof of Lemma~\ref{adh16.1.1}.
Set $\beta_n \coloneqq v_L(f_n-b_n)-\alpha_{n+1}$.
The asymptotic couple $(\Gamma_{K\langle f \rangle}, \psi)$ of $K\langle f \rangle$ has the following properties:
\begin{enumerate}
    \item\label{adh16.1.2i} $(\beta_n)$ is $\Q$-linearly independent over $\Gamma$ and $\Gamma_{K\langle f \rangle} = \Gamma \oplus \bigoplus_n \Z\beta_n$ {\normalfont(}internal direct sum{\normalfont)};
    \item\label{adh16.1.2ii} $(\beta_n^\dagger)$ is $\Q$-linearly independent over $\Gamma$, so $\beta_n^\dagger \notin \Gamma$ and $\beta_m^\dagger \neq \beta_n^\dagger$ for all $m \neq n$;
    \item\label{adh16.1.2iii} $\psi(\Gamma_{K\langle f \rangle}^{\neq}) = \Psi \cup \{ \beta_n^\dagger : n \in \N \}$;
    \item\label{adh16.1.2iv} $[\beta_n] \notin [\Gamma]$ for all $n$, $[\beta_m] \neq [\beta_n]$ for all $m \neq n$, and  $[\Gamma_{K\langle f \rangle}] = [\Gamma] \cup \{ [\beta_n] : n \in \N \}$;
    \item\label{adh16.1.2v} if $\Gamma^<$ is cofinal in $\Gamma_{K\langle f \rangle}^<$, then $\beta_0^\dagger < \beta_1^\dagger < \beta_2^\dagger < \cdots$.
\end{enumerate}
\end{lem}
\begin{proof}
The $\Q$-linear independence of $(\beta_n)$ over $\Gamma$ follows from the proof of Lemma~\ref{adh16.1.1}.
Items \ref{adh16.1.2i}--\ref{adh16.1.2iv} are proved exactly as in \cite[Lemma~16.1.2]{adamtt}, without modification, but are sketched below.
The only difference is \ref{adh16.1.2v}: In \cite[Lemma~16.1.2]{adamtt}, cofinality is established unconditionally. Here, Lemma~\ref{bm:2.8} immediately gives \ref{adh16.1.2v}.

The hardest part is \ref{adh16.1.2i}, which involves showing that the element $P(f) \in K\langle f \rangle$, where $P \in K\{Y\}^{\neq}$, can be expressed as a polynomial in the ``monomials'' $\f m_n \coloneqq (f_n-b_n)/a_{n+1}$.
Since $v(\f m_n) = \beta_n$, item \ref{adh16.1.2i} then follows from the $\Q$-linear independence of $(\beta_n)$ over $\Gamma$.
With that done, the $\Q$-linear independence of $(\beta_n^{\dagger})$ over $\Gamma$ follows from the $\Q$-linear independence of $(\beta_n)_{n \ges 1}$ over $\Gamma$ by noticing that $\beta_n^{\dagger} = \beta_{n+1}+\alpha_{n+2}$, and the rest of \ref{adh16.1.2ii}--\ref{adh16.1.2iv} follows.
\end{proof}



\subsection{Further consequences in the ordered setting}

Now we develop further the results of the previous subsection in the pre-$H$-field setting.
If $K$, $L$, and $f$ are as before but additionally pre-$H$-fields, and $M$ is a pre-$H$-field with $g \in M$ realizing the same cut in $K$ as $f$, our goal is to embed $K\langle f \rangle$ into $M$ over $K$ by sending $f$ to $g$.
This is accomplished in Proposition~\ref{realizecutembed} and Lemma~\ref{notcofinal}.
The former handles the case that $\Gamma^<$ is cofinal in $\Gamma_{K\langle f \rangle}^<$, which follows from the previous subsection by making small adjustments to the arguments for analogous results in \cite[\S16.1]{adamtt}.
The case that $\Gamma^<$ is not cofinal in $\Gamma_{K\langle f \rangle}^<$ does not occur in that context, and here is handled in Lemma~\ref{notcofinal} by similar ideas.

In the next lemma, the reader may assume without harm that $K$, $L$, and $f$ additionally satisfy the assumptions of Lemma~\ref{adh16.1.1}, but the lemma is stated in more generality to clarify which assumptions are used.
In particular, it does not assume that $K$ and $L$ have gap~$0$ or that $K$ is $\d$-henselian.
Small modifications to the proof of \cite[Lemma~16.1.4]{adamtt} yield the following generalization.

\begin{lem}\label{realizecut}
Let $K$ be a pre-$H$-field with exponential integration and let $L$ be a pre-$H$-field extension of $K$ with $\bm k_L = \bm k$.
Suppose that $f \in L \setminus K$ and $b_0 \in K$ satisfy:
\begin{enumerate}
    \item $v_L(f-K) \subseteq \Gamma_L$ has a maximum and $v_L(f-b_0) = \max v_L(f-K)$;
    \item $\Gamma^<$ is cofinal in $\Gamma_{K\langle f \rangle}^<$.
\end{enumerate}
Suppose that $M$ is a pre-$H$-field extension of $K$ and $g \in M$ realizes the same cut in $K$ as~$f$.
Then $v_M(g-b_0) = \max v_M(g-K) \notin \Gamma$ and $(g-b_0)^\dagger$ realizes the same cut in $K$ as $(f-b_0)^{\dagger}$.
\end{lem}
\begin{proof}
Let $\alpha \in \Gamma$ and $b \in K$.
First we use the convexity of $\ca O_{K\langle f \rangle}$ and the cofinality assumption to show that
\[
v_L(f-b)<\alpha \iff v_M(g-b)<\alpha \qquad \text{and} \qquad v_L(f-b)>\alpha \iff v_M(g-b)>\alpha.
\]
To see this, take $a \in K^>$ with $va = \alpha$.
Suppose that $v_L(f-b)<\alpha$, so $|f-b|>a$.
Hence $|g-b|>a$, and so $v_M(g-b)\les \alpha$.
To get the strict inequality, use the cofinality assumption to take $\delta \in \Gamma$ with $v_L(f-b)<\delta<\alpha$, and thus $v_M(g-b) \les \delta < \alpha$.
One proves similarly that $v_L(f-b)>\alpha \implies v_M(g-b)>\alpha$.
Finally, consider the case that $v_L(f-b)=\alpha$.
This yields $f-b \sim ua$ for $u \in K$ with $u \asymp 1$, since $\bm k = \bm k_L$.
Convexity yields $|u|a/2 < |f-b| < 2|u|a$, so $|u|a/2 < |g-b| < 2|u|a$, and thus $v_M(g-b)=va=\alpha$, completing the proof of the displayed equivalences.
Hence, $v_M(g-b_0) \notin \Gamma$, since $v_L(f-b_0) \notin \Gamma$.
This in turn yields $v_M(g-b_0)=\max v_M(g-K)$.
It also follows that $(g-b_0)^\dagger \notin K$, as otherwise $(g-b_0)^\dagger = b^\dagger$ for some $b \in K^{\x}$, in which case $v_M(g-b_0)=vb \in \Gamma$.

As written, the argument in \cite[Lemma~16.1.4]{adamtt} that $(g-b_0)^\dagger$ realizes the same cut in $K$ as $(f-b_0)^\dagger$ uses that $K$ and $L$ are $H$-fields, so we give the following modified version.
First, we may assume that $f>b_0$, so $g>b_0$.
Suppose towards a contradiction that we have $h \in K$ with $(f-b_0)^\dagger<h$ and $h<(g-b_0)^\dagger$.
Take $\phi \in K^>$ with $h=\phi^\dagger$ and set $s \coloneqq (f-b_0)/\phi$.
Then we have $s>0$ and $s^\dagger=(f-b_0)^\dagger-h<0$.
Since $L$ is a pre-$H$-field, $v_L(s) \ges 0$ by \cite[Lemma~10.5.2(i)]{adamtt}, but since $v_L(f-b_0) \notin \Gamma$, we have $v_L(s)>0$.
Thus $0<s<1$.
Similarly, $h<(g-b_0)^\dagger$ gives $t \coloneqq (g-b_0)/\phi>0$ and $t^\dagger>0$, so $v_M(t)<0$ and thus $t>1$.
Putting this together yields
\[
f\ =\ b_0+\phi s\ <\ b_0+\phi\ \qquad \text{and}\ \qquad b_0+\phi\ <\ b_0 + \phi t\ =\ g,
\]
contradicting that $f$ and $g$ realize the same cut in $K$.
The other case, that there is $h \in K$ with $(f-b_0)^\dagger>h$ and $h>(g-b_0)^\dagger$, is symmetric.
\end{proof}

The next embedding lemma follows from Lemma~\ref{realizecut} in the same way that \cite[Proposition~16.1.5]{adamtt} follows from \cite[Lemma~16.1.4]{adamtt}, without modification.
The idea is to use Lemma~\ref{realizecut} to carry out the proof of Lemma~\ref{adh16.1.2} in $M$ and thereby construct the desired embedding.

\begin{prop}\label{realizecutembed}
Let $K$, $L$, and $f$ be as in Lemma~\ref{adh16.1.1} and suppose that $\Gamma^<$ is cofinal in $\Gamma_{K\langle f \rangle}^<$.
Additionally, suppose that $K$ and $L$ are pre-$H$-fields and $L$ is a pre-$H$-field extension of $K$ with gap~$0$.
Suppose that $M$ is a pre-$H$-field extension of $K$ with gap~$0$ and $g \in M$ realizes the same cut in $K$ as $f$.
Then there exists an embedding $K \langle f \rangle \to M$ over $K$ with $f \mapsto g$.
\end{prop}

Above we assumed that $\Gamma^<$ was cofinal in $\Gamma_{K\langle f \rangle}^<$, and now we turn to the other case, which does not occur in the $H$-field setting of \cite[\S16.1]{adamtt}.

\begin{lem}\label{notcofinal}
Let $K$ be a pre-$H$-field and $L$ be a pre-$H$-field extension of $K$ with gap~$0$.
Let $f \in L^>$ with $\Gamma^< < v_L(f) < 0$.
Suppose that $M$ is a pre-$H$-field extension of $K$ with gap~$0$ and $g \in M^>$ satisfies $\Gamma^< < v_M(g) < 0$.
Then there is an embedding $K\langle f \rangle \to M$ over $K$ with $f \mapsto g$.
\end{lem}
\begin{proof}
Set $f_0 \coloneqq f$ and $f_{n+1} \coloneqq f_n^\dagger$, and let $\beta_n \coloneqq v_L(f_n) \in \Gamma_L$.
Then
\[
[\Gamma^{\neq}] > [\beta_0]>[\beta_1]>[\beta_2]>\cdots>[0].
\]
In particular, $[\beta_n] \notin [\Gamma]$ for all $n$ and $\beta_0, \beta_1, \beta_2, \dots$ are $\Q$-linearly independent over $\Gamma$.
Hence the vector space $\Q\Gamma_{K\langle f \rangle}/\Gamma$ is infinite dimensional, so $f$ is $\d$-transcendental over $K$.
Since $f=f_0$ and $f_n' = f_n f_{n+1}$, induction yields $F_n \in K[Y_0, \dots, Y_n]$ of (total) degree at most $n+1$ with $f^{(n)}=F_n(f_0, \dots, f_n)$.
Let $P \in K\{Y\}^{\neq}$ of order at most $r$.
Then $P(f)=\sum_{\bm i \in I} a_{\bm i} f_0^{i_0} \dots f_r^{i_r}$, where $I$ is a nonempty finite set of indices $\bm i=(i_0, \dots, i_r) \in \N^{1+r}$.
In particular, $\Gamma_{K\langle f \rangle} = \Gamma \oplus \bigoplus_{n} \Z\beta_n$.

Set $g_0 \coloneqq g$, $g_{n+1} \coloneqq g_n^\dagger$, and $\beta_n^* \coloneqq v_M(g_n) \in \Gamma_M$.
The same argument yields that $g$ is $\d$-transcendental over $K$ and $P(g)=\sum_{\bm i \in I} a_{\bm i} g_0^{i_0} \dots g_r^{i_r}$, where $I$ is the same set of indices as in $P(f)$ and $a_{\bm i}$ are the same coefficients.
Hence $\Gamma_{K\langle g \rangle} = \Gamma \oplus \bigoplus_{n} \Z\beta_n^*$.
Thus we have an isomorphism of ordered abelian groups $j \colon \Gamma_{K\langle f \rangle} \to \Gamma_{K\langle g \rangle}$ with $\beta_n \mapsto \beta_n^*$.
By the expressions for $P(f)$ and $P(g)$, we have $j\big(v_L(P(f))\big) = v_M(P(g))$, which yields a valued differential field embedding from $K\langle f \rangle \to M$ over $K$ with $f \mapsto g$.
To see that this is an ordered valued differential field embedding, note that $f_n>0$ and $g_n>0$ for all $n$, so $P(f)>0 \iff P(g)>0$.
\end{proof}

\section{Extending the constant field}\label{sec:rcc}
The results of this short section are used mainly to strengthen the statement of Theorem~\ref{rcexpintclosure} in the next section and play no role in the main results of \S\ref{sec:main}.
The term ``residue constant closed'' defined in the next paragraph, however, is used in some key lemmas of \S\ref{sec:dhl}.
\begin{ass}In this section, $K$ is asymptotic with small derivation.\end{ass}
Since $C \subseteq \ca O$, $C$ maps injectively into $\bm k$ under the residue field map, and hence into $C_{\bm k}$.
We say that $K$ is \deft{residue constant closed} if  $K$ is henselian and $C$ maps onto $C_{\bm k}$, that is, $\res(C)=C_{\bm k}$.
We say that $L$ is a \deft{residue constant closure} of $K$ if it is a residue constant closed $H$-asymptotic extension of $K$ with small derivation that embeds uniquely over $K$ into every residue constant closed $H$-asymptotic extension $M$ of $K$ with small derivation.
Note that if $K$ has a residue constant closure, then it is unique up to unique isomorphism over $K$.
\begin{prop}\label{rcclosure}
Suppose that $K$ is pre-$\d$-valued of $H$-type with $\sup\Psi=0$. Then $K$ has a residue constant closure that is an immediate extension of $K$.
\end{prop}
\begin{proof}
Recall from \S\ref{sec:ac-small} that $\sup\Psi=0 \iff (\Gamma^>)'=\Gamma^>$.
Also note that if $L$ is an immediate asymptotic extension of $K$, then it is $H$-asymptotic, satisfies $\Psi_L=\Psi$ (in particular, $\sup\Psi_L=0$), and is pre-$\d$-valued by \cite[Corollary~10.1.17]{adamtt}.

Build a tower of immediate asymptotic extensions of $K$ as follows.
Set $K_0 \coloneqq K$.
If $K_\lambda$ is not henselian, set $K_{\lambda+1} \coloneqq K_\lambda^{\h}$, the henselization of $K_\lambda$, which as an algebraic extension of $K_\lambda$ is asymptotic \cite[Proposition~9.5.3]{adamtt}.
If $K_\lambda$ is residue constant closed, we are done, so suppose that $K_\lambda$ is henselian but not residue constant closed.
Then we have $u \in K_\lambda$ with $u \asymp 1$, $u'\prec 1$, and $u' \notin \der\cao_{K_\lambda}$.
Let $y$ be transcendental over $K_\lambda$ and equip $K_{\lambda+1} \coloneqq K_\lambda(y)$ with the unique derivation extending that of $K_\lambda$ such that $y'=u'$.
Since $\sup\Psi_{K_{\lambda}}=0$, we have $v(u') \in (\Gamma^>)'$, so the set $\{ v(u'-a') : a \in \cao_{K_\lambda} \}$ has no maximum \cite[Lemma~10.2.5(iii)]{adamtt}.
Thus by \cite[Lemma~10.2.4]{adamtt} we can equip $K_{\lambda+1}$ with the unique valuation making it an $H$-asymptotic extension of $K_\lambda$ with $y \not\asymp 1$; with this valuation, $y \prec 1$ and $K_{\lambda+1}$ is an immediate extension of $K_{\lambda}$.
If $\lambda$ is a limit ordinal, set $K_\lambda \coloneqq \bigcup_{\rho<\lambda} K_\rho$.
Since each extension is immediate, by Zorn's lemma we may take a maximal such tower $(K_\lambda)_{\lambda\les\mu}$.

It is clear that $K_\mu$ is residue constant closed, and we show that it also has the desired universal property.
Let $M$ be an $H$-asymptotic extension of $K$ with small derivation that is residue constant closed, and let $\lambda<\mu$ and $i \colon K_{\lambda} \to M$ be an embedding.
It suffices by induction to extend $i$ uniquely to an embedding $K_{\lambda+1} \to M$.
If $K_{\lambda+1}=K_{\lambda}^{\h}$, then this follows from the universal property of the henselization.
Now suppose that $K_{\lambda+1} = K_\lambda(y)$ with $y$ and $u$ as above.
Take the unique $c \in C_M$ with $c \sim i(u)$ and set $z \coloneqq i(u)-c$.
Then $z'=i(u)'$ and $z \prec 1$, so by the remarks after \cite[Lemma~10.2.4]{adamtt}, $z$ is transcendental over $i(K_\lambda)$, and thus mapping $y \mapsto z$ yields a differential field embedding $K_{\lambda+1} \to M$ extending $i$.
By the uniqueness of \cite[Lemma~10.2.4]{adamtt}, this is a valued differential field embedding.
Finally, if $i$ extends to an embedding with $y \mapsto z_1 \in M$, then $i(u)-z_1 \in C_M$ and $i(u)-z_1 \sim i(u) \sim c$, so $z_1=z$.
\end{proof}
Note that if $K$ is a pre-$H$-field with $\sup\Psi=0$, then as an immediate extension of $K$ any residue constant closure of $K$ embeds uniquely, as an \emph{ordered} valued differential field, over $K$ into every residue constant closed pre-$H$-field extension of $K$ with small derivation.

The following result is worth recording but not needed later.
In it, we construe an algebraic extension of a residue constant closed $K$ as a valued differential field extension of $K$ by equipping it with the unique derivation and valuation extending those of $K$ (uniqueness of the valuation follows from the henselianity of $K$); if $K$ is additionally an ordered field, then by henselianity $\ca O$ is necessarily convex.
\begin{lem}\label{rcclosedalg}
Suppose that $K$ is residue constant closed.
An algebraic extension $L$ of $K$ remains residue constant closed.
\end{lem}
\begin{proof}
First, recall that an algebraic extension of a henselian valued field is henselian, and note that $L$ is asymptotic \cite[Proposition~9.5.3]{adamtt} and has small derivation \cite[Proposition~6.2.1]{adamtt}.
Let $u \in L$ with $u \asymp 1$ and $u' \prec 1$; we need to show that there is $c \in C_{L}$ with $c \sim u$.
Since $\overline{u} \in C_{\res(L)}$ is algebraic over $\res(K)$, it is algebraic over $C_{\res(K)}$.
Take a $P \in C[X]$ such that $\overline{P} \in C_{\res(K)}[X]$ is the minimum polynomial of $\overline{u}$ over $C_{\res(K)}$.
Then $P = Q\cdot\prod_{i=1}^n (X-c_i)$, where $c_i \in C_{L}$ for $i=1,\dots,n$ and $Q \in C_{L}[X]$ has no roots in $C_{L}$.
Moreover, $Q$ has no roots in $L$, so henselianity yields $i$ with $1 \les i \les n$ and $\overline{c_i}=\overline{u}$.
\end{proof}

\section{Differential-Hensel-Liouville closures}\label{sec:dhl}
\begin{ass}In this section, $K$ is a pre-$H$-field.\end{ass}
In this section we construct differential-Hensel-Liouville closures (Theorem~\ref{dhlclosure}) in analogy with the Newton-Liouville closures of \cite[\S14.5]{adamtt} and prove that they are unique (Corollary~\ref{dhlclosuremin}).
To do this, we first construct extensions that are real closed, have exponential integration, and satisfy an embedding property (Corollary~\ref{expintclosure}, Lemma~\ref{expintclosureembed}), in analogy with the Liouville closures of \cite{ad-hf}; however, we use \cite[\S10.6]{adamtt} as a reference and also adapt some preliminaries from \cite[\S10.4--10.6]{adamtt}.
Combining this with the residue constant closures from the previous section, we record in Theorem~\ref{rcexpintclosure} an improvement that is not needed later.

Suppose momentarily that $K$ has gap~$0$.
In several results below, it is assumed that $\bm k$ has exponential integration.
It is worth noting that if $K$ has exponential integration, then so does $\bm k$.
Hence, the assumption that $\bm k$ has exponential integration is necessary for $K$ to have an extension with exponential integration and the same ordered differential residue field.

\subsection{Adjoining exponential integrals}\label{sec:dhl:adexpint}
Suppose that $s \in K \setminus (K^{\x})^\dagger$ and $f$ is transcendental over $K$.
We give $K(f)$ the unique derivation extending that of $K$ with $f^\dagger = s$.
In the first lemma, $K$ need only be an ordered differential field.

\begin{lem}\label{rcsameconstants}
If $K$ is real closed and $K(f)$ can be ordered making it an ordered field extension of $K$, then $C_{K(f)} = C$.
\end{lem}
\begin{proof}
Suppose towards a contradiction that we have $y \in K^{\x}$ and $m \ges 1$ with $y^{\dagger} = ms$.
Then by taking a $z \in K^{\x}$ with $z^m=y$, after arranging $y>0$ if necessary, we have $z^{\dagger}=s$, a contradiction.
Thus \cite[Lemma~4.6.11]{adamtt} (with $R=K$, $r=s$, and $x=f$) and \cite[Corollary~4.6.12]{adamtt} (with $R=K[f, f^{-1}]$) yield that $C_{K(f)}$ is algebraic over $C$, so $C_{K(f)} = C$.
\end{proof}

In the next two lemmas, $K$ is just a valued differential field, and need not be ordered.
The first is based on \cite[Lemma~10.4.2]{adamtt}.
\begin{lem}\label{adh10.4.2}
Suppose that $K$ has small derivation and $\bm k = (\bm k^{\x})^\dagger$.
Let $K(f)$ have a valuation that makes it an extension of $K$ with $\Gamma_{K(f)} = \Gamma$ and $\der\ca O_{K(f)} \subseteq \ca O_{K(f)}$.
Then $s-a^\dagger \prec 1$ for some $a \in K^{\x}$.
\end{lem}
\begin{proof}
Since $vf \in \Gamma$, there is $b \in K^{\x}$ with $g \coloneqq f/b \asymp 1$.
Then $s-b^\dagger = g^\dagger \asymp g' \prece 1$.
If $s-b^\dagger \prec 1$, set $a \coloneqq b$.
If $s-b^\dagger \asymp 1$, since $\bm k = (\bm k^{\x})^\dagger$, we have $u \asymp 1$ in $K^{\x}$ with $s-b^\dagger \sim u^\dagger$.
Then set $a \coloneqq bu$.
\end{proof}

The last part of the previous argument also yields the following useful fact.
\begin{lem}\label{10.4.2cor}
Suppose that $K$ has small derivation and $\bm k = (\bm k^{\x})^\dagger$.
If $s-a^\dagger \succe 1$ for all $a \in K^{\x}$, then $s-a^\dagger \succ 1$ for all $a \in K^{\x}$.
\end{lem}

Now we return to the situation that $K$ is a pre-$H$-field.
\begin{lem}[{\cite[Lemma~10.5.18]{adamtt}}]\label{adh10.5.18}
Suppose that $K$ is henselian and $vs \in (\Gamma^>)'$.
Then there is a unique valuation on $K(f)$ making it an $H$-asymptotic extension of $K$ with $f \sim 1$.
With this valuation, $K(f)$ is an immediate extension of $K$, so there is a unique ordering of $K(f)$ making it a pre-$H$-field extension of $K$.
\end{lem}

Here is a pre-$H$-field version of \cite[Lemma~10.5.20]{adamtt} with the same proof; the assumption there that $K$ is an $H$-field (rather than a pre-$H$-field) is only used at the end of the proof in showing that $L$ too is an $H$-field.
Recall that $\Psi^{\downarrow}$ is the downward closure of $\Psi$ in~$\Gamma$.
\begin{lem}\label{adh10.5.20}
Suppose that $K$ is real closed, $s<0$, and $v(s-a^\dagger) \in \Psi^{\downarrow}$ for all $a \in K^{\x}$. 
Then there is a unique pair of a field ordering and a valuation on $L \coloneqq K(f)$ making it a pre-$H$-field extension of $K$ with $f>0$.
Moreover, we have:
\begin{enumerate}
    \item $vf \notin \Gamma$, $\Gamma_L = \Gamma \oplus \Z vf$, $f \prec 1$;
    \item $\Psi$ is cofinal in $\Psi_L \coloneqq \psi_L(\Gamma_L^{\neq})$;
    \item a gap in $K$ remains a gap in $L$;
    \item if $L$ has a gap not in $\Gamma$, then $[\Gamma_L]=[\Gamma]$;
    \item $\bm k_L = \bm k$.
\end{enumerate}
\end{lem}

\subsection{Exponential integration closures}\label{sec:dhl:expint}
Let $E$ be a differential field.
We call a differential field extension $F$ of $E$ an \deft{exponential integration extension} of $E$ (\deft{expint-extension} for short) if $C_F$ is algebraic over $C_E$ and for every $a \in F$ there are $t_1, \dots, t_n \in F^{\x}$ with $a \in E(t_1, \dots, t_n)$ such that for $i=1,\dots,n$, either $t_i$ is algebraic over $E(t_1,\dots,t_{i-1})$ or $t_i^\dagger \in E(t_1,\dots,t_{i-1})$.
In particular, any expint-extension is $\d$-algebraic.
The following is routine.

\begin{lem}\label{adh10.6.6o}
Let $E \subseteq F \subseteq M$ be a chain of differential field extensions.
\begin{enumerate}
    \item If $M$ is an expint-extension of $E$, then $M$ is an expint-extension of $F$.
    \item If $M$ is an expint-extension of $F$ and $F$ is an expint-extension of $E$, then $M$ is an expint-extension of $E$.
\end{enumerate}
\end{lem}

\begin{lem}\label{adh10.6.8}
If $F$ is an expint-extension of $E$, then $|F|=|E|$.
\end{lem}
\begin{proof}
Supposing that $F$ is an expint-extension of $E$, define an increasing sequence of differential subfields of $F$ by setting $E_0 \coloneqq E$ and $E_{n+1}$ to be the algebraic closure of $E_n$ in $F$ when $n$ is even and $E_{n+1} \coloneqq E_n(\{ b \in F^{\times} : b^{\dagger} \in E_n\})$ when $n$ is odd.
Clearly, $F = \bigcup_n E_n$, so it remains to check by induction that $|E_n|=|E|$ for all~$n$.
\end{proof}

Now suppose that $E$ is an ordered differential field.
We call $E$ \deft{exponential integration closed} (\deft{expint-closed} for short) if it is real closed and has exponential integration.
We call an ordered differential field extension $F$ of $E$ an \deft{exponential integration closure} (\deft{expint-closure} for short) of $E$ if it is an expint-extension of $E$ that is expint-closed.
Note that in this latter definition, ``closure'' does not include a universal property, unlike for instance the gap-closures of \S\ref{sec:ac-small}.

The next observation has the same proof as \cite[Lemma~10.6.9]{adamtt}.
\begin{lem}\label{adh10.6.9o}
If $E$ is expint-closed, then $E$ has no proper expint-extension with the same constants as~$E$.
\end{lem}

\begin{ass}For the rest of this subsection, suppose that $K$ has gap~$0$.\end{ass}
From this assumption it follows that $(\Gamma^>)'=\Gamma^>$ and $\Psi^{\downarrow} = \Gamma^<$ (see \S\ref{sec:ac-small}).

Below, we construe the real closure $K^{\rc}$ of $K$ as an ordered valued differential field by equipping it with the unique derivation extending that of $K$ and the unique valuation extending that of $K$ whose valuation ring is convex (see for example \cite[Corollary~3.5.18]{adamtt}).
Then $\Gamma_{K^{\rc}}$ is the divisible hull $\Q\Gamma$ of $\Gamma$, $\bm k_{K^{\rc}}$ is the real closure of $\bm k$, and $C_{K^{\rc}}$ is the real closure of $C$.
Since $K$ is a pre-$H$-field, so is $K^{\rc}$ by \cite[Proposition~10.5.4]{adamtt}, and it also has gap~$0$ (see \S\ref{sec:ac-small} for details on the extensions of the asymptotic couples).

\begin{defn}
We call a strictly increasing chain $(K_\lambda)_{\lambda \les \mu}$ of pre-$H$-fields with gap~$0$ an \deft{expint-tower on $K$} if:
\begin{enumerate}
    \item $K_0=K$;
    \item if $\lambda$ is a limit ordinal, then $K_{\lambda} =  \bigcup_{\rho<\lambda} K_\rho$;
    \item if $\lambda<\lambda+1\les \mu$, then either:
        \begin{enumerate}
            \myitem[(a)]\manuallabel{expinttower:a}{(a)} $K_\lambda$ is not real closed and $K_{\lambda+1}$ is the real closure of $K_\lambda$; or
            \myitem[(b)]\manuallabel{expinttower:b}{(b)} $K_\lambda$ is real closed and $K_{\lambda+1}=K_{\lambda}(y_\lambda)$ with $y_\lambda \notin K_\lambda$ satisfying either:
                \begin{enumerate}
                    \myitem[(b1)]\manuallabel{expinttower:b1}{(b1)} $y_\lambda^\dagger = s_\lambda \in K_\lambda$ with $y_\lambda \sim 1$, $s_\lambda \prec 1$, and $s_\lambda \neq a^\dagger$ for all $a \in K_\lambda^{\x}$; or
                    \myitem[(b2)]\manuallabel{expinttower:b2}{(b2)} $y_\lambda^\dagger = s_\lambda \in K_\lambda$ with $s_\lambda<0$, $y_\lambda>0$, and $s_\lambda-a^\dagger \succ 1$ for all $a \in K_\lambda^{\x}$.
                \end{enumerate}
        \end{enumerate}
\end{enumerate}
Given such a tower, we call $K_{\mu}$ its  \deft{top} and set $C_\lambda \coloneqq C_{K_\lambda}$ and $\bm k_\lambda \coloneqq \bm k_{K_\lambda}$ for $\lambda\les\mu$.
\end{defn}

\begin{lem}\label{adh10.6.13o}
Let $(K_\lambda)_{\lambda \les \mu}$ be an expint-tower on $K$.
Then:
\begin{enumerate}
    \item\label{adh10.6.13o1} $K_\mu$ is an expint-extension of $K$;
    \item\label{adh10.6.13o2} $C_\mu$ is the real closure of $C$ if $\mu>0$;
    \item\label{adh10.6.13o3} $\bm k_\mu$ is the real closure of $\bm k$ if $\mu>0$;
    \item\label{adh10.6.13o4} $|K_\lambda|=|K|$, and hence $\mu<|K|^+$.
\end{enumerate}
\end{lem}
\begin{proof}
For \ref{adh10.6.13o1}, go by induction on $\lambda\les\mu$.
The main thing to check is the condition on the constant fields.
If $\lambda=0$ or $\lambda$ is a limit ordinal, this is clear.
If $K_{\lambda+1}$ is the real closure of $K_\lambda$, then $C_{\lambda+1}$ is the real closure of $C_{\lambda}$.
If $K_\lambda$ is real closed and $K_{\lambda+1}$ is as in \ref{expinttower:b} above, then $C_{\lambda+1}=C_{\lambda}$ by Lemma~\ref{rcsameconstants}.

For \ref{adh10.6.13o2}, $C_1$ is the real closure of $C$, and then $C_\lambda=C_1$ for all $\lambda \ges 1$ as in the proof of \ref{adh10.6.13o1}.

For \ref{adh10.6.13o3}, $\bm k_1$ is the real closure of $\bm k$, and then $\bm k_\lambda = \bm k_1$ for all $\lambda \ges 1$;
in case \ref{expinttower:a} this is obvious, in case \ref{expinttower:b1} 
Lemma~\ref{adh10.5.18} applies, and in case \ref{expinttower:b2} Lemma~\ref{adh10.5.20} applies.

Finally, \ref{adh10.6.13o4} follows from \ref{adh10.6.13o1} and Lemma~\ref{adh10.6.8}.
\end{proof}

\begin{lem}\label{adh10.6.14o}
Let $L$ be the top of a maximal expint-tower on $K$ such that $\bm k_L$ has exponential integration.
Then $L$ is an expint-closure of $K$.
\end{lem}
\begin{proof}
Suppose that $L$ is not expint-closed.
If $L$ is not real closed, then its real closure is a proper pre-$H$-field extension of $L$ with gap~$0$, so we could extend the expint-tower.
We are left with the case that $L$ is real closed and we have $s \in L \setminus (L^{\x})^\dagger$.
In particular, $L$ is henselian.
We may assume that $s<0$.
Take $f$ transcendental over $L$ with $f^\dagger = s$.

First suppose that $s-a^\dagger \prec 1$ for some $a \in L^{\x}$.
Then taking such an $a$ and replacing $f$ and $s$ by $f/a$ and $s-a^\dagger$, we arrange that $s \prec 1$.
Giving $L(f)$ the valuation and ordering from Lemma~\ref{adh10.5.18} makes it a pre-$H$-field extension of $L$ with gap~$0$ of type~\ref{expinttower:b1}.

Now suppose that $s-a^\dagger \succe 1$ for all $a \in L^{\x}$.
By Lemma~\ref{10.4.2cor}, $s-a^\dagger \succ 1$ for all $a \in L^{\x}$.
Then giving $L(f)$ the ordering and valuation from Lemma~\ref{adh10.5.20} makes it a pre-$H$-field extension of $L$ with gap~$0$ of type~\ref{expinttower:b2}.

Thus $L$ is expint-closed, and hence an expint-closure of $K$ by Lemma~\ref{adh10.6.13o}\ref{adh10.6.13o1}.
\end{proof}

Recall from \S\ref{sec:rcc} the term ``residue constant closed''.
In a few arguments below, this property roughly plays the same role as ``$\d$-valued'' does in the construction of Liouville closures in \cite[\S10.6]{adamtt}.
\begin{cor}\label{expintclosure}
Suppose that $\bm k$ is expint-closed.
Then $K$ has an expint-closure that is a pre-$H$-field extension of $K$ with gap~$0$.
If $K$ is residue constant closed, then $K$ has a residue constant closed expint-closure that is a pre-$H$-field extension of $K$ with gap~$0$.
\end{cor}
\begin{proof}
By Lemma~\ref{adh10.6.13o}\ref{adh10.6.13o4}, Zorn gives a maximal expint-tower $(K_\lambda)_{\lambda \les \mu}$ on $K$.
Then $\bm k_\mu = \bm k$ by Lemma~\ref{adh10.6.13o}\ref{adh10.6.13o3}, so $C_{\bm k_\mu}=C_{\bm k}$, and $C_{\mu}=C$ by Lemma~\ref{adh10.6.13o}\ref{adh10.6.13o2}.
Hence $K_\mu$ is an expint-closure of $K$ by Lemma~\ref{adh10.6.14o}.
If $K$ is residue constant closed, then so is $K_\mu$ since it is real closed, so henselian, and $C=C_{\mu}$ maps onto $C_{\bm k} = C_{\bm k_\mu}$.
\end{proof}

\begin{lem}\label{adh10.6.15o}
Let $M$ be a residue constant closed, expint-closed pre-$H$-field extension of $K$ with gap~$0$.
Suppose that $(K_\lambda)_{\lambda\les\mu}$ is an expint-tower on $K$ in $M$ \textnormal{(}i.e., each $K_\lambda$ is a pre-$H$-subfield of $M$\textnormal{)} and maximal in $M$ \textnormal{(}i.e., it cannot be extended to an expint-tower $(K_\lambda)_{\lambda\les\mu+1}$ on $K$ in $M$\textnormal{)} such that $\bm k_\mu$ has exponential integration.
Then $(K_\lambda)_{\lambda\les\mu}$ is a maximal expint-tower on $K$.
\end{lem}
\begin{proof}
Since $M$ is real closed, $K_\mu$ must be real closed by maximality in $M$.
So supposing $(K_\lambda)_{\lambda \les \mu}$ is not a maximal expint-tower on $K$, we have $s_\mu \in K_\mu$ such that $s_\mu \neq a^\dagger$ for all $a \in K_\mu^{\x}$; we may assume that $s_\mu<0$. 
Since $M$ is expint-closed, we have $y_\mu \in M$ with $y_\mu^\dagger=s_\mu$; we may assume that $y_\mu>0$.

First suppose that $s_\mu-a^\dagger \succe 1$ for all $a \in K_\mu^{\x}$, so actually $s_\mu-a^\dagger \succ 1$ for all $a \in K_\mu^{\x}$ by Lemma~\ref{10.4.2cor}.
Thus setting $K_{\mu+1} \coloneqq K_\mu(y_\mu)$ yields an extension of $(K_\lambda)_{\lambda \les \mu}$ in $M$ of type \ref{expinttower:b2}.

Now suppose that $s_\mu-a^\dagger \prec 1$ for some $a \in K_\mu^{\x}$.
Taking such an $a$ and replacing $s_\mu$ and $y_\mu$ by $s_\mu-a^\dagger$ and $y_\mu/a$, we may assume that $s_\mu \prec 1$.
Since $M$ has gap~$0$, we have $y_\mu \asymp 1$ and so $y_\mu' \asymp s_\mu \prec 1$.
That is, $\overline{y_\mu} \in C_{\res(M)}$, so we have $c \in C_M$ with $y_\mu \sim c$.
Replacing $y_\mu$ by $y_\mu/c$, we obtain the desired extension of $(K_\lambda)_{\lambda \les \mu}$ in $M$ of type \ref{expinttower:b1}.
\end{proof}
This comment is not used later, but in the above lemma, we can replace the assumption that $M$ is residue constant closed (so $C_{\res(M)} = \res(C_{M})$) with $C_{\res(M)} = C_{\res(K)}$.
In the final argument, instead of $c \in C_M$ we have $u \in K$ with $u \asymp 1$ and $u' \prec 1$, so we also replace $s_\mu$ with $s_\mu-u^\dagger$.

\begin{cor}\label{expintclosuremin}
Suppose that $L$ is an expint-closed pre-$H$-field extension of~$K$.
\begin{enumerate}
    \item\label{expintclosuremin1} If $L$ is an expint-closure of $K$, then no proper differential subfield of $L$ containing $K$ is expint-closed.
    \item\label{expintclosuremin2} Suppose that $\bm k$ is expint-closed, and that $L$ has gap~$0$ and is residue constant closed. If no proper differential subfield of $L$ containing $K$ is expint-closed, then $L$ is an expint-closure of~$K$.
\end{enumerate}
\end{cor}
\begin{proof}
For \ref{expintclosuremin1}, if $L$ is an expint-closure of $K$, then no proper differential subfield of $L$ containing $K$ is expint-closed by Lemmas~\ref{adh10.6.6o} and \ref{adh10.6.9o}.

For \ref{expintclosuremin2}, suppose that no proper differential subfield of $L$ containing $K$ is expint-closed.
Take an expint-tower $(K_\lambda)_{\lambda\les\mu}$ on $K$ in $L$ that is maximal in $L$.
Since $\bm k$ is real closed, $\bm k_\mu=\bm k$, and hence $\bm k_\mu$ has exponential integration.
Then $(K_\lambda)_{\lambda\les\mu}$ is a maximal expint-tower on $K$ by Lemma~\ref{adh10.6.15o}.
By Lemma~\ref{adh10.6.14o}, $K_\mu$ is expint-closed and hence equal to~$L$.
\end{proof}

\begin{lem}\label{expintclosureembed}
Let $(K_\lambda)_{\lambda\les\mu}$ be an expint-tower on $K$.
Then any embedding of $K$ into a residue constant closed, expint-closed pre-$H$-field extension $M$ of $K$ with gap~$0$ extends to an embedding of $K_\mu$.
If $K$ is residue constant closed and $\bm k$ is expint-closed, then any two residue constant closed expint-closures of $K$ that are pre-$H$-field extensions of $K$ with gap~$0$ are isomorphic over $K$.
\end{lem}
\begin{proof}
Let $M$ be a residue constant closed, expint-closed pre-$H$-field with gap~$0$.
We prove that for $\lambda<\mu$ any embedding $K_\lambda \to M$ extends to an embedding $K_{\lambda+1} \to M$, which yields the result by induction.
Suppose that $i \colon K_\lambda \to M$ is an embedding.
If $K_{\lambda+1}$ is the real closure of $K_\lambda$, then we may extend $i$ to~$K_{\lambda+1}$.

So suppose that $K_\lambda$ is real closed and we have $s_\lambda \in K_{\lambda}$ and $y_\lambda \in K_{\lambda+1}\setminus K_{\lambda}$ with $K_{\lambda+1}=K_{\lambda}(y_\lambda)$, $y_\lambda^\dagger = s_\lambda$, $y_\lambda \sim 1$, $s_\lambda \prec 1$, and $s_\lambda \neq a^\dagger$ for all $a \in K_\lambda^{\x}$.
Take $z \in M$ with $z^\dagger = i(s_\lambda)$.
Hence $z \asymp 1$ and $\overline{z} \in C_{\res(M)}$, so we have $c \in C_M$ with $z \sim c$.
By the uniqueness of Lemma~\ref{adh10.5.18}, we may extend $i$ to an embedding of $K_\lambda(y_\lambda)$ into $M$ sending $y_\lambda$ to~$z/c$.

Now suppose that $K_\lambda$ is real closed and we have $s_\lambda \in K_\lambda$ and $y_\lambda \in K_{\lambda+1}\setminus K_\lambda$ with $K_{\lambda+1}=K_{\lambda}(y_\lambda)$, $y_\lambda^\dagger = s_\lambda$, $s_\lambda<0$, $y_\lambda>0$, and $s_\lambda-a^\dagger \succ 1$ for all $a \in K^{\x}_\lambda$.
Take $z \in M$ with $z^\dagger = i(s_\lambda)$; we may assume that $z>0$.
Then by the uniqueness of Lemma~\ref{adh10.5.20}, we can extend $i$ to an embedding of $K_\lambda(y_\lambda)$ into $M$ sending $y_\lambda$ to~$z$.

The second statement follows from the first by Corollary~\ref{expintclosuremin}\ref{expintclosuremin1} and the proof of Corollary~\ref{expintclosure}.
\end{proof}

Combining these results with Proposition~\ref{rcclosure} yields the following theorem, although for the construction of differential-Hensel-Liouville closures in Theorem~\ref{dhlclosure}, Corollary~\ref{expintclosure} and Lemma~\ref{expintclosureembed} already suffice.
\begin{thm}\label{rcexpintclosure}
Suppose that $\bm k$ is expint-closed.
Then $K$ has a pre-$H$-field extension $L$ with gap~$0$ such that:
\begin{enumerate}
	\item\label{rcexpintclosure1} $L$ is a residue constant closed, expint-closed extension of~$K$;
	\item\label{rcexpintclosure2} $L$ embeds over $K$ into any residue constant closed, expint-closed pre-$H$-field extension of $K$ with gap~$0$;
	\item\label{rcexpintclosure3} $L$ has no proper differential subfield containing $K$ that is residue constant closed and expint-closed.
\end{enumerate}
\end{thm}
\begin{proof}
By Proposition~\ref{rcclosure}, let $K_0$ be the residue constant closure of $K$.
Taking the top of a maximal expint-tower on $K_0$ yields an expint-closure $L$ of $K_0$ that is residue constant closed as in Corollary~\ref{expintclosure}.
For \ref{rcexpintclosure2}, let $M$ be a pre-$H$-field extension of $K$ with gap~$0$ that is residue constant closed and expint-closed.
Then $K_0$ embeds uniquely into $M$ over $K$, so by Lemma~\ref{expintclosureembed} we can extend this to an embedding of $L$.
For \ref{rcexpintclosure3}, suppose that $L_0 \supseteq K$ is a differential subfield of $L$ that is residue constant closed and expint-closed.
Then $L_0 \supseteq K_0$, and hence $L_0=L$ by Corollary~\ref{expintclosuremin}\ref{expintclosuremin1}.
\end{proof}
Let $L$ be as above.
Then any pre-$H$-field extension of $K$ with gap~$0$ satisfying \ref{rcexpintclosure1} and \ref{rcexpintclosure2} is isomorphic to $L$ over~$K$ by  \ref{rcexpintclosure3}.
Also, $L$ is a Liouville extension\footnote{The definition is similar to that of expint-extensions in \S\ref{sec:dhl:expint} except that we also allow $t_i' \in E(t_1, \dots, t_{i-1})$.} of $K$, since $K_0$ is a Liouville extension of $K$ by construction and expint-extensions are Liouville extensions.

\subsection{Differential-Hensel-Liouville closures}
\begin{ass}We continue to assume in this subsection that the pre-$H$-field $K$ has gap~$0$.\end{ass}
\begin{defn}
We call $K$ \deft{differential-Hensel-Liouville closed} (slightly shorter: \deft{$\d$-Hensel-Liouville closed}) if it is $\d$-henselian and expint-closed.
We call a pre-$H$-field extension $L$ of $K$ a \deft{differential-Hensel-Liouville closure} (slightly shorter: \deft{$\d$-Hensel-Liouville closure}) of $K$ if it is $\d$-Hensel-Liouville closed and embeds over $K$ into every $\d$-Hensel-Liouville closed pre-$H$-field extension of~$K$.
\end{defn}
Note that if $K$ is $\d$-henselian, then $K$ is also closed under integration in the sense that $\der$ is surjective \cite[Lemma~7.1.8]{adamtt}, hence the use of ``Liouville'' in the terms just defined.

To build $\d$-Hensel-Liouville closures, we use the fact that if $F$ is an asymptotic valued differential field with small derivation and linearly surjective differential residue field, then it has a (unique) \deft{differential-henselization} (\deft{$\d$-henselization} for short) $F^{\dh}$ by \cite[Theorem~3.7]{pc-dh}.
For such $F$, $F^{\dh}$ is an immediate asymptotic $\d$-algebraic extension of $F$ that is $\d$-henselian and embeds over $F$ into every $\d$-henselian asymptotic extension of $F$; if $F$ is a pre-$H$-field, then $F^{\dh}$ is too and embeds (as an \emph{ordered} valued differential field) into every $\d$-henselian pre-$H$-field extension of~$F$.

\begin{thm}\label{dhlclosure}
Suppose that $\bm k$ is expint-closed and linearly surjective.
Then $K$ has a $\d$-Hensel-Liouville closure~$K^{\dhl}$.
\end{thm}
\begin{proof}
We use below that any $\d$-henselian asymptotic field is residue constant closed by \cite[Lemma~9.4.10]{adamtt}.
Define a sequence of pre-$H$-field extensions of $K$ with gap~$0$ as follows.
Set $K_0 \coloneqq K$.
For $n \ges 1$, if $n$ is odd, let $K_n$ be the $\d$-henselization of $K_{n-1}$, and if $n$ is even, let $K_n$ be the expint-closure of $K_{n-1}$ from Corollary~\ref{expintclosure}.
Note that $\bm k_{K_n} = \bm k$ for all $n$.
We set $K^{\dhl} \coloneqq \bigcup_{n} K_n$ and show that $K^{\dhl}$ is a $\d$-Hensel-Liouville closure of~$K$.

Let $L$ be a pre-$H$-field extension of $K$ that is $\d$-henselian and expint-closed.
We show by induction on $n$ that we can extend any embedding $K_n \to L$ to an embedding $K_{n+1} \to L$, so suppose that we have an embedding $i \colon K_n \to L$.
If $n$ is even, then $K_{n+1}$ is the $\d$-henselization of $K_n$, so we may extend $i$ to an embedding $K_{n+1} \to L$.
If $n$ is odd, then $K_n$ is $\d$-henselian and $K_{n+1}$ is the expint-closure of $K_n$, so we can extend $i$ to an embedding $K_{n+1} \to L$ by Lemma~\ref{expintclosureembed}.
\end{proof}

Note that $K^{\dhl}$ is a $\d$-algebraic extension of $K$ with the same residue field.
In the next two results, adapted from \cite[\S16.2]{adamtt}, we show that $K^{\dhl}$ is the unique, up to isomorphism over $K$, $\d$-Hensel-Liouville closure of~$K$.

\begin{lem}\label{dhldalg}
Suppose that $\bm k$ is expint-closed and linearly surjective.
Let $i \colon K^{\dhl} \to L$ be an embedding into a pre-$H$-field $L$ with gap~$0$ such that $\res\big(i(K^{\dhl})\big)=\res(L)$.
Then
\[
i(K^{\dhl})\ =\ i(K)^{\dalg}\ \coloneqq\ \{ f \in L : f\ \text{is $\d$-algebraic over}\ i(K) \}.
\]
\end{lem}
\begin{proof}
We have $i(K^{\dhl}) \subseteq i(K)^{\dalg}$ since $K^{\dhl}$ is a $\d$-algebraic extension of $K$.
For the other direction, note that $i(K^{\dhl})$ is a $\d$-henselian, expint-closed pre-$H$-subfield of $i(K)^{\dalg}$, so $i(K^{\dhl})=i(K)^{\dalg}$ by Theorem~\ref{dalgmaxupgrade}.
\end{proof}
Hence for $K$ as in the lemma above, any $\d$-algebraic extension of $K$ that is a $\d$-henselian, expint-closed pre-$H$-field with the same residue field as $K$ is isomorphic to $K^{\dhl}$ over $K$, and is thus a $\d$-Hensel-Liouville closure of~$K$.

\begin{cor}\label{dhlclosuremin}
Suppose that $\bm k$ is expint-closed and linearly surjective.
Then $K^{\dhl}$ has no proper differential subfield containing $K$ that is $\d$-Hensel-Liouville closed.
Thus any $\d$-Hensel-Liouville closure of $K$ is isomorphic to $K^{\dhl}$ over~$K$.
\end{cor}
\begin{proof}
If $L \supseteq K$ is a $\d$-Hensel-Liouville closed differential subfield of $K^{\dhl}$, then we have an embedding $i \colon K^{\dhl} \to L$ over $K$.
Viewing this as an embedding into $K^{\dhl}$, by Lemma~\ref{dhldalg} we have $K^{\dhl} = i(K^{\dhl})$, so $K^{\dhl}=L$.
\end{proof}

\section{Main results}\label{sec:main}
\begin{ass}
In this section, $K$ is a pre-$H$-field with gap~$0$.
\end{ass}
\subsection{Two-sorted results}\label{sec:maintwo}
We first combine earlier results to establish the key embedding lemma.
For an ordered set $S$ we denote the cofinality of $S$ by $\cf(S)$.
In \textit{Case~2} of the next lemma, recall from \S\ref{sec:16.1} the brief discussion of pc-sequences and the notion ``$\d$-algebraically maximal'', along with its connection to pc-sequences.

\begin{lem}\label{qe:sameres}
Suppose that $K$ is $\d$-Hensel-Liouville closed, and let $E$ be a pre-$H$-subfield of $K$ with $\res(E) = \res(K)$.
Let $L$ be a $\d$-Hensel-Liouville closed pre-$H$-field such that $L$ is $|K|^+$-saturated as an ordered set and $\cf(\Gamma_L^<)>|\Gamma|$.
Then any embedding $E \to L$ can be extended to an embedding $K \to L$.
\end{lem}
\begin{proof}
Let $i \colon E \to L$ be an embedding.
We may assume that $E \neq K$.
It suffices to show that $i$ can be extended to an embedding $F \to L$ for some pre-$H$-subfield $F$ of $K$ properly containing~$E$.

First, suppose that $\Gamma_E^<$ is not cofinal in $\Gamma^<$ and let $f \in K^>$ with $\Gamma_E^< < vf < 0$.
By the cofinality assumption on $\Gamma_L^<$, take $g \in L^>$ with $\Gamma_{i(E)}^< < v_L(g) < 0$.
Then we extend $i$ to an embedding $E \langle f \rangle \to L$ sending $f \mapsto g$ by Lemma~\ref{notcofinal}.

Now suppose that $\Gamma_E^<$ is cofinal in $\Gamma^<$ and consider the following three cases.

\textit{Case~1: $E$ is not $\d$-Hensel-Liouville closed.}
From the assumptions on $K$, we get that $\res(K)$ is expint-closed and linearly surjective.
Since $\res(E) = \res(K)$, we may extend $i$ to an embedding of the $\d$-Hensel-Liouville closure of $E$ into $L$ by Theorem~\ref{dhlclosure}.

\textit{Case~2: $E$ is $\d$-Hensel-Liouville closed and $E\langle y \rangle$ is an immediate extension of $E$ for some $y \in K \setminus E$.}
Take such a $y$ and let $(a_\rho)$ be a divergent pc-sequence in $K$ with $a_\rho \pconv y$.
Since $E$ is $\d$-henselian, it is $\d$-algebraically maximal \cite[Theorem~3.6]{pc-dh}, and so $(a_\rho)$ is of $\d$-transcendental type over $E$.
By the saturation assumption on $L$ and \cite[Lemma~2.4.2]{adamtt}, we have $z \in L$ with $i(a_\rho) \pconv z$.
Then \cite[Lemma~6.9.1]{adamtt} yields a valued differential field embedding $E\langle y \rangle \to L$ sending $y \mapsto z$.
Since $E\langle y \rangle$ is an immediate extension of $E$, this is also an ordered field embedding.

\textit{Case~3: $E$ is $\d$-Hensel-Liouville closed and there is no $y \in K \setminus E$ with $E\langle y \rangle$ an immediate extension of $E$.}
Take any $f \in K \setminus E$.
By saturation, take $g \in L$ such that for all $a \in E$, we have
\[
a<f \implies i(a)<g \qquad \text{and} \qquad f<a \implies g<i(a).
\]
Then we can extend $i$ to an embedding $E\langle f \rangle \to L$ with $f \mapsto g$ by Proposition~\ref{realizecutembed}.
\end{proof}

We now suspend the convention that $\bm k$ is the ordered differential residue field of $K$, and instead let $\bm k$ be an ordered differential field such that we have a surjective ordered differential ring homomorphism $\pi \colon \ca O \to \bm k$; note that $\pi$ has kernel $\cao$, so it induces an ordered differential field isomorphism $\res(K) \to \bm k$.
Thus we continue to think of $\bm k$ as the ordered differential residue field of $K$, but below possibly equipped with additional structure.
We consider two-sorted structures of the form $\bm K = (K, \bm k; \pi_2)$ in the language $\ca L_2$, where the language on the sort $K$ is $\ca L = \{+, -, \cdot, 0, 1, \der, \prece, \les\}$, the language on the sort $\bm k$ is $\ca L_{\res} \supseteq \{+, -, \cdot, 0, 1, \der, \les\}$, and $\pi_2 \colon K^2 \to \bm k$ is defined by:
\begin{enumerate}
	\item $\pi_2(a, b) = \pi(ab^{-1})$ for $a \in K$ and $b \in K^{\x}$ with $a \prece b$;
	\item $\pi_2(a, b) = 0$ for $a,b \in K$ with $a \succ b$;
	\item $\pi_2(0, 0) = 0$.
\end{enumerate}
Note that the dominance relation $\prece$ can be axiomatized as a binary relation on $K$, without reference to $v$ or $\Gamma$ (see for example \cite[Definition~3.1.1]{adamtt} and successive remarks).
Also, the only reason for including a binary residue map here instead of the unary residue map is so we can avoid including multiplicative inversion in the language.

Fix an $\ca L_{\res}$-theory $T_{\res}$ extending the theory of ordered differential fields that are linearly surjective and expint-closed, and let $T^{\dhl}$ be the $\ca L_2$-theory whose models are structures $\bm K = (K, \bm k; \pi_2)$ as above such that $K$ is $\d$-Hensel-Liouville closed, $\Gamma\neq\{0\}$, and $\bm k \models T_{\res}$.

\subsubsection{Equivalence Theorem}
Let $\bm K = (K, \bm k; \pi_2)$ and $\bm K^* = (K^*, \bm k^*; \pi_2^*)$ be models of $T^{\dhl}$.
We aim to construct a back-and-forth system from $\bm K$ to $\bm K^*$ (when $\bm K$ and $\bm K^*$ are sufficiently saturated).
To that end, a \deft{good substructure} of $\bm K$ is an $\ca L_2$-substructure $\bm E = (E, \bm k_{\bm E})$ of $\bm K$ such that $E$ and $\bm k_E$ are fields, where we drop the $\pi_2$ from the notation.
Note that then $E$ is a pre-$H$-subfield of $K$ with gap~$0$ and also that the restriction $\pi_2 \colon E^2 \to \bm k_E$ need not be surjective.
Let $\bm E$ and $\bm E^*$ be good substructures of $\bm K$ and $\bm K^*$, respectively.
We say that a map $\bm f \colon \bm E \to \bm E^*$ is a \deft{good map} if $\bm f = (f, f_{\r})$ is an $\ca L_2$-isomorphism such that $f_{\r} \colon \bm k_{\bm E} \to \bm k_{\bm E^*}$ is moreover elementary as a partial map from $\bm k$ to $\bm k^*$.
In particular, $f \colon E \to E^*$ is a pre-$H$-field isomorphism and $f_{\r} \colon \bm k_{\bm E} \to \bm k_{\bm E^*}$ is an $\ca L_{\res}$-isomorphism.

\begin{thm}[Equivalence Theorem]\label{eqthm}
Every good map $\bm E \to \bm E^*$ is elementary as a partial map from $\bm K$ to~$\bm K^*$.
\end{thm}
\begin{proof}
Let $\bm f = (f, f_{\r})$ be a good map from $\bm E$ to $\bm E^*$.
Let $\kappa$ be a cardinal of uncountable cofinality such that $\max\{|E|, |\bm k_{\bm E}|, |\ca L_{\res}|\}<\kappa$ and $2^\lambda < \kappa$ for every cardinal $\lambda<\kappa$. 
By passing to elementary extensions, we may suppose that $\bm K$ and $\bm K^*$ are $\kappa$-saturated.
We say a good substructure $(E_1, \bm k_1)$ of $\bm K$ is \deft{small} if $\max\{|E_1|, |\bm k_1|\}<\kappa$; note that if $(E_1, \bm k_1)$ is small and $E_2 \subseteq K$ is a pre-$H$-field extension of $E_1$ with $\Gamma_{E_2}=\Gamma_{E_1}$ and $\pi_2(E_2, E_2) \subseteq \bm k_1$, then $(E_2, \bm k_1)$ is small.
To establish the theorem, it suffices to show that the set of good maps with small domain is a back-and-forth system from $\bm K$ to $\bm K^*$.

Given $a \in K$, we need to extend $\bm f$ to a good map with small domain $(F, \bm k_{\bm F})$ such that $a \in F$.
First, we note two basic procedures for extending $\bm k_{\bm E}$ and $\pi_2(E, E)$:
\begin{enumerate}
	\item\label{bnf1} Given $d \in \bm k$, arranging that $d \in \bm k_{\bm E}$: By the saturation assumption, we can extend $f_{\r}$ to a partial elementary map with $d$ in its domain. Without changing $f$ or $E$, this yields an extension of $\bm f$ to a good map with small domain.
	\item\label{bnf2} Given $d \in \bm k_{\bm E} \setminus \pi_2(E, E)$, arranging that $d \in \pi_2(E, E)$: By Lemmas~\ref{adh7.1.4o} and \ref{adh6.3.1o}, we have $b \in K$ such that $\pi_2(E\langle b \rangle, E\langle b \rangle) = \pi_2(E, E)\langle d \rangle$ and $\Gamma_{E\langle b \rangle} = \Gamma_E$, and a good map $(g, f_{\r})$ extending $\bm f$ with small domain $(E\langle b \rangle, \bm k_{\bm E})$.
\end{enumerate}
In this proof only, we call an ordered differential field \deft{closed} if it is real closed, linearly surjective, and has exponential integration.
Consider $E \langle a \rangle$.
By \ref{bnf1}, we extend $f_{\r}$ to a partial elementary map $f_{1, \r}$ with domain $\bm k_1$ such that $\bm k_1$ is closed, $\pi_2(E \langle a \rangle, E \langle a \rangle) \subseteq \bm k_1$, and $|\bm k_1|<\kappa$.
By \ref{bnf2}, we extend $f$ to $f_1$ so that $\bm f_1 = (f_1, f_{1, \r})$ is a good map with small domain $(E_1, \bm k_1)$ satisfying $\pi_2(E_1, E_1) = \bm k_1$.
In the same way, we extend $\bm f_1$ to a good map $\bm f_2 = (f_2, f_{2, \r})$ with small domain $(E_2, \bm k_2)$ such that $\bm k_2$ is closed and $\pi_2(E_1 \langle a \rangle, E_1 \langle a \rangle) \subseteq \bm k_2 = \pi_2(E_2, E_2)$.
Iterating this procedure and taking unions yields a good map $\bm f_\omega = (f_\omega, f_{\omega, \r})$ with small domain $\bm E_\omega = (E_\omega, \bm k_\omega)$ such that $\bm k_\omega$ is closed and $\pi_2(E_\omega \langle a \rangle, E_\omega \langle a \rangle) = \bm k_\omega = \pi_2(E_\omega, E_\omega)$.

Now take the $\d$-Hensel-Liouville closure $F \subseteq K$ of $E_\omega\langle a \rangle$ by Theorem~\ref{dhlclosure}.
By construction, $\pi_2(F,F)=\bm k_\omega$ and $F = \bigcup_n F_n$ is a countable increasing union of pre-$H$-fields $F_n$ such that $F_0 = E_\omega\langle a \rangle$ and for all $n$, either $|F_{n+1}|=|F_n|$ or $F_{n+1}$ is an immediate extension of $F_n$.
It follows that $(F, \bm k_\omega)$ is small.
Thus we may apply Lemma~\ref{qe:sameres} to extend $f_\omega$ to $f_{\omega+1}$ so that $(f_{\omega+1}, f_{\omega, \r})$ is a good map with small domain $(F, \bm k_\omega)$.
\end{proof}

\begin{cor}\label{ake}
We have $\bm K \equiv \bm K^*$ if and only if $\bm k \equiv \bm k^*$.
\end{cor}
\begin{proof}
The left-to-right direction is trivial, so suppose that $\bm k \equiv \bm k^*$.
We may assume that $\ca L_{\res}$ is an expansion of $\{+, -, \cdot, 0, 1, \der, \les\}$ by relation symbols, so then we can identify $\Q$ with an $\ca L_{\res}$-substructure of $\bm k$ and an $\ca L_{\res}$-substructure of $\bm k^*$, respectively, and by assumption, these are isomorphic.
Consider $(\Q, \Q; \pi_2)$: the first $\Q$ is endowed with its usual ordered ring structure, the trivial derivation, and the trivial dominance relation $\prece$; the second $\Q$ is an $\ca L_{\res}$-structure as above;
and $\pi_2 \colon \Q^2 \to \Q$ is defined by, for $q_1, q_2 \in \Q$, $\pi_2(q_1, q_2) = q_1q_2^{-1}$ if $q_2 \neq 0$ and $\pi_2(q_1, 0) = 0$.
This structure embeds into both $\bm K$ and $\bm K^*$, inducing an obvious good map between good substructures, which is then elementary as a partial map from $\bm K$ to $\bm K^*$ by Theorem~\ref{eqthm}.
Hence $\bm K \equiv \bm K^*$.
\end{proof}

\begin{cor}\label{multimodcompl}
Let $\bm E = (E, \bm k_E; \pi_2) \subseteq \bm K$ with $\bm E \models T^{\dhl}$ and $\bm k_E \prece \bm k$.
Then $\bm E \prece \bm K$.
\end{cor}
\begin{proof}
View the identity map on $(E, \bm k_E)$ as a map from $\bm E$ to~$\bm K$ and note that it is good.
\end{proof}
In particular, if $T_{\res}$ is model complete, then so is $T^{\dhl}$.
We now prove relative quantifier elimination, using much more of the strength of the Equivalence Theorem than the previous corollaries did.

\subsubsection{Relative quantifier elimination}
Let $x = (x_1, \dots, x_m)$ be a tuple of distinct variables of sort $K$ and $y = (y_1, \dots, y_n)$ be a tuple of distinct variables of sort $\bm k$.
We call an $\ca L_2$-formula $\psi(x, y)$ \deft{special} if $\psi(x, y)$ is
\[
\psi_{\r}\big(\pi_2(P_1(x), Q_1(x)), \dots, \pi_2(P_{\ell}(x), Q_{\ell}(x)), y\big),
\]
for some $\ell \in \N$, $\ca L_{\res}$-formula $\psi_{\r}(u_1, \dots, u_{\ell}, y)$, and $P_1, Q_1, \dots, P_{\ell}, Q_{\ell} \in \Z\{X_1, \dots, X_m\}$.

\begin{thm}\label{relQE}
Let $\phi(x, y)$ be an $\ca L_2$-formula with $x$ and $y$ as above.
Then $\phi(x, y)$ is $T^{\dhl}$-equivalent to
\begin{equation}\tag{$*$}\label{relQEform}
\big(\theta_1(x) \wedge \psi_1(x,y)\big) \vee \dots \vee \big(\theta_N(x) \wedge \psi_N(x,y)\big)
\end{equation}
for some $N \in \N$, quantifier-free $\ca L$-formulas $\theta_1(x)$, \dots, $\theta_N(x)$, and special formulas $\psi_1(x,y)$, \dots, $\psi_N(x,y)$.
\end{thm}
\begin{proof}
Let $\Theta(x, y)$ be the set of $\ca L_2$-formulas displayed in \eqref{relQEform}.
Then $\Theta(x, y)$ is obviously closed under disjunction and also closed under negation, up to logical equivalence.
It suffices to show that every $(x,y)$-type (consistent with $T^{\dhl}$) is determined by its intersection with $\Theta(x, y)$.
Below, $\theta(x)$ ranges over quantifier-free $\ca L$-formulas and $\psi(x, y)$ ranges over special formulas.
For a model $\bm K = (K, \bm k; \pi_2)$ of $T^{\dhl}$ and $a \in K^{m}$ and $d \in \bm k^{n}$, we set
\[
\qftp^K(a)\ \coloneqq\ \{ \theta(x) : K \models \theta(a) \}
\]
and
\[
\tp_{\r}^{\bm K}(a, d)\ \coloneqq\ \{ \psi(x, y) : \bm K \models \psi(a, d) \}.
\]
Let $\bm K = (K, \bm k; \pi_2)$ and $\bm K^* = (K^*, \bm k^*; \pi_2^*)$ be models of $T^{\dhl}$.
Let $a \in K^{m}$, $a^* \in (K^*)^{m}$, $d \in \bm k^{n}$, and $d^* \in (\bm k^*)^{n}$, and assume that $\qftp^K(a) = \qftp^{K^*}(a^*)$ and $\tp_{\r}^{\bm K}(a, d) = \tp_{\r}^{\bm K^*}(a^*, d^*)$.
We need to show that $\tp^{\bm K}(a, d) = \tp^{\bm K^*}(a^*, d^*)$.

Let $E \coloneqq \Q\langle a \rangle$ be the differential subfield of $K$ generated by $a$, construed as a pre-$H$-subfield of $K$, and let $\bm k_{\bm E}$ be the $\ca L_{\res}$-substructure and subfield of $\bm k$ generated by $\pi_2(E,E)$ and $d$.
Then $\bm E = (E, \bm k_{\bm E})$ is a good substructure of $\bm K$.
Defining $E^*$ and $\bm k_{\bm E^*}$ likewise yields a good substructure $\bm E^* = (E^*, \bm k_{\bm E^*})$ of $\bm K^*$.
Since we have $\qftp^{K}(a) = \qftp^{K^*}(a^*)$ and $\tp_{\r}^{\bm K}(a, d) = \tp_{\r}^{\bm K^*}(a^*, d^*)$, the natural map $\bm E \to \bm E^*$ sending $a \mapsto a^*$ and $d \mapsto d^*$ is good, and hence $\tp^{\bm K}(a, d) = \tp^{\bm K^*}(a^*, d^*)$ by Theorem~\ref{eqthm}.
\end{proof}

Note that the outer $\ca L_{\res}$-formula in a special formula may have quantifiers, so the previous theorem eliminates quantifiers down to quantifiers over the sort $\bm k$.
In particular, if $T_{\res}$ has quantifier elimination, then so does $T^{\dhl}$.
Also, $\bm k$ is \emph{purely stably embedded} in $\bm K$ in the following sense:
\begin{cor}\label{stabembed}
Let $\bm K = (K, \bm k; \pi_2) \models T^{\dhl}$.
Every subset of $\bm k^n$ definable in $\bm K$ \textnormal{(}with parameters from $\bm K$\textnormal{)} is definable in the $\ca L_{\res}$-structure $\bm k$ \textnormal{(}with parameters from~$\bm k$\textnormal{)}.
\end{cor}

Let $T'_{\res} \subseteq T_{\res}$ and let $T$ be the $\ca L_2$-theory of structures $(K, \bm k; \pi_2)$ such that $K$ is a pre-$H$-field with gap~$0$, $\bm k \models T'_{\res}$, and we no longer require $\pi \colon \ca O \to \bm k$ to be surjective.
\begin{cor}\label{multimodcompa}
If $T_{\res}$ is the model companion of $T'_{\res}$, then $T^{\dhl}$ is the model companion of~$T$.
If $T_{\res}$ additionally has quantifier elimination, then $T^{\dhl}$ is the model completion of~$T$.
\end{cor}
\begin{proof}
Suppose that every model of $T'_{\res}$ can be extended to a model of $T_{\res}$.
By Corollary~\ref{multimodcompl} and Theorem~\ref{relQE}, it remains to show that every model of $T$ can be extended to a model of $T^{\dhl}$.
Let $(K, \bm k; \pi_2) \models T$.
First extend $\bm k$ to $\bm k^* \models T_{\res}$ and apply Corollary~\ref{adh6.3.3o} to obtain a pre-$H$-field extension $K^*$ of $K$ with gap~$0$ and a surjective ordered differential ring homomorphism $\pi^* \colon \ca O_{K^*} \to \bm k^*$.
Now by Theorem~\ref{dhlclosure} extend $K^*$ to its $\d$-Hensel-Liouville closure $(K^*)^{\dhl}$ and $\pi^*$ to $\pi^* \colon \ca O_{(K^*)^{\dhl}} \to \bm k^*$, and note that $\big((K^*)^{\dhl}, \bm k^*; \pi_2^*\big) \models T^{\dhl}$.
\end{proof}

The NIP transfer principle in Proposition~\ref{NIPtransfer} below follows from Theorem~\ref{relQE} and an analogous transfer principle from \cite{delon,belairbousquet} by the standard trick of ``forgetting'' the derivation.
In those papers, it is shown that the three-sorted theory of henselian valued fields of equicharacteristic~$0$ has NIP if the residue field does.
More explicitly, let $(K, \bm k, \Gamma; \pi, v)$ be a valued field of equicharacteristic~$0$ construed as a three-sorted structure.
That is, $K$ is a field in the language of rings, $\Gamma$ is an ordered abelian group, and $\bm k$ is a field of characteristic~$0$ in the language of rings such that the surjective map $v \colon K^{\x} \to \Gamma$ makes $K$ a henselian valued field with residue field (isomorphic to) $\bm k$ and residue field map $\pi \colon \ca O \to \bm k$.
The relevant NIP transfer principle from \cite{delon,belairbousquet} is that if $\bm k$ has NIP, then so does $(K, \bm k, \Gamma; \pi, v)$.
Moreover, although not explicitly stated there, the proof goes through when $\bm k$ has extra structure.
It is not relevant here, but one can also expand $(K, \bm k, \Gamma; \pi, v)$ by an angular component map, as done in \cite{belairbousquet}.

What is the difference with our setup? We have $\bm K = (K, \bm k; \pi_2) \models T^{\dhl}$, a two-sorted structure in which $K$ is additionally equipped with a derivation, an ordering, and a dominance relation, and we have a binary rather than unary version of the residue field map.
Although $\bm k$ too is equipped with additional structure, this is taken care of already, as mentioned above.
Let $K_{\ring}$ be the reduct of $K$ to the language of rings.
Then the ordering, the dominance relation, and $\pi_2$ are definable without parameters in $(K_{\ring}, \bm k, \Gamma; \pi, v)$:
Since $K$ is real closed, its ordering is definable in $K_{\ring}$, and its dominance relation is definable from $v$ and $\Gamma$.
Finally, $\pi_2$ is definable from $\pi$.
Thus, we must only deal with the derivation, as done in detail in the proposition.

\begin{prop}\label{NIPtransfer}
Let $\bm K = (K, \bm k; \pi_2) \models T^{\dhl}$.
If the $\ca L_{\res}$-structure $\bm k$ has NIP, then so does~$\bm K$.
\end{prop}
\begin{proof}
Suppose that $\bm k$ has NIP.
Recall that a boolean combination of formulas that have NIP has NIP, so by Theorem~\ref{relQE} it suffices to show that every quantifier-free $\ca L$-formula has NIP and every special formula has NIP.

First, consider a special formula $\psi(x; y, z)$, where $y$ is an $m$-tuple of distinct variables of sort $K$, $z$ is an $n$-tuple of distinct variables of sort $\bm k$, and $x$ is a single variable distinct from the variables in $y$ and $z$, and suppose towards a contradiction that $\psi(x; y, z)$ is independent in $\bm K$.
As a reminder, this means that the family defined by $\psi(x; y, z)$ has infinite VC-dimension, where, by convention, the semicolon indicates that $x$ is the object variable and $y,z$ are the parameter variables.
That is, for every $N \in \N$ with $N \ges 1$, there exist $a_1, \dots, a_N$ of the same sort as $x$ such that, for all $I \subseteq \{1, \dots, N\}$, there exist $b_I \in K^{m}$ and $d_I \in \bm k^{n}$ with:
\[
\bm K \models \psi(a_i; b_I, d_I)\ \iff\ i \in I.
\]

First suppose that the variable $x$ is of sort $\bm k$.
As a special formula, $\psi(x; y, z)$ is thus \[\psi_{\r}\big(x; \pi_2(P_1(y), Q_1(y)), \dots, \pi_2(P_{\ell}(y), Q_{\ell}(y)), z\big),\]
where $\ell \in \N$, $P_1, Q_1, \dots, P_{\ell}, Q_{\ell} \in \Z\{Y_1, \dots, Y_m\}$, and $\psi_{\r}(x; u_1, \dots, u_{\ell}, z)$ is an $\ca L_{\res}$-formula.
Then $\psi_{\r}(x; u_1, \dots, u_{\ell}, z)$ itself is independent in $\bm k$.
To see this, let $N \ges 1$ and $I \subseteq \{1, \dots, N\}$, and take $a_1, \dots, a_N \in \bm k$, $b_I \in K^{m}$, and $d_I \in \bm k^{n}$ as above.
Then letting \[e_I\ \coloneqq\ \big(\pi_2(P_1(b_I), Q_1(b_I)), \dots, \pi_2(P_{\ell}(b_I), Q_{\ell}(b_I))\big)\ \in\ \bm k^{\ell},\] we have
\begin{align*}
\bm k \models \psi_{\r}(a_i; e_I, d_I)\ &\iff\ \bm K \models \psi_{\r}\big(a_i; \pi_2(P_1(b_I), Q_1(b_I)), \dots, \pi_2(P_{\ell}(b_I), Q_{\ell}(b_I)), d_I\big)\\
&\iff\ i \in I.
\end{align*}

Now suppose that $x$ is of sort $K$, so then $\psi(x; y, z)$ is \[\psi_{\r}\big(\pi_2(P_1(x, y), Q_1(x, y)), \dots, \pi_2(P_{\ell}(x, y), Q_{\ell}(x, y)), z\big),\]
where $\ell \in \N$, $P_1, Q_1, \dots, P_{\ell}, Q_{\ell} \in \Z\{X, Y_1, \dots, Y_m\}$, and $\psi_{\r}(u_1, \dots, u_{\ell}, z)$ is an $\ca L_{\res}$-formula.
For notational simplicity, we assume that $\ell=m=n=1$ and replace $P_1$ by $P$, $Q_1$ by $Q$, and $Y_1$ by $Y$.
It is no longer obvious how to extract an independent $\ca L_{\res}$-formula from $\psi$ now that $x$ and $y$ both appear in the term $\pi_2(P(x,y),Q(x,y))$, so for this case we instead appeal to the results of \cite{delon,belairbousquet} (this works too for the previous case, but the goal in separating the cases is to illustrate the reason for using \cite{delon,belairbousquet}).
Now comes the ``forgetting'' of the derivation.
Explicitly, take $r$ and $p, q \in \Z[X_0, \dots, X_r, Y_{0}, \dots, Y_{r}]$ such that
\[P(a, b)\ =\ p(a, a', \dots, a^{(r)}, b, b', \dots, b^{(r)})\qquad \text{and}\qquad Q(a, b)\ =\ q(a, a', \dots, a^{(r)}, b, b', \dots, b^{(r)})\]
for all $a, b \in K$.
Then for all $a,b \in K$ and $d \in \bm k$, we have $\bm K \models \psi(a; b, d)$ if and only if
\[
(K_{\ring}, \bm k; \pi_2) \models \psi_{\r}\big(\pi_2(p(a, \dots, a^{(r)}, b, \dots, b^{(r)}), q(a, \dots, a^{(r)}, b, \dots, b^{(r)})), d\big),
\]
where $K_{\ring}$ is the reduct of $K$ to the language of rings (no change is made here to the sort $\bm k$, which in particular retains its derivation).
Letting $\varphi(x_0, \dots, x_r; y_0, \dots, y_r, z)$ be
\[ \psi_{\r}\big(\pi_2(p(x_0, \dots, x_r, y_0, \dots, y_r), q(x_0, \dots, x_r, y_0, \dots, y_r)), z\big), \]
$\varphi$ is independent in $(K_{\ring}, \bm k; \pi_2)$, but since $\pi_2$ is definable in $(K_{\ring}, \bm k, \Gamma; \pi, v)$, this yields an independent formula in $(K_{\ring}, \bm k, \Gamma; \pi, v)$, contradicting \cite{delon,belairbousquet}.

A similar argument shows that every quantifier-free $\ca L$-formula has NIP; alternatively, reduce to the fact that RCVF, the theory of real closed fields with a nontrivial valuation whose valuation ring is convex, has NIP, bypassing \cite{delon,belairbousquet}.
\end{proof}

\subsection{One-sorted results}\label{sec:mainone}

Recall that the theory of closed ordered differential fields is the model completion of the theory of ordered differential fields, where no interaction is assumed between the ordering and the derivation.
This theory has quantifier elimination and is complete \cite{singer-codf}.
We have the following results for $T^{\dhl}_{\codf}$, the one-sorted theory of $\d$-Hensel-Liouville closed pre-$H$-fields with nontrivial value group and closed ordered differential residue field.
\begin{thm}\label{preH:qe}
The theory $T^{\dhl}_{\codf}$:
\begin{enumerate}
    \item\label{preH:qeitem} has quantifier elimination;
    \item\label{preH:modcompitem} is the model completion of the theory of pre-$H$-fields with gap~$0$;
    \item\label{preH:compitem} is complete, and hence decidable;
    \item\label{preH:distal} is distal, and hence has NIP;
    \item\label{preH:locomin} is locally o-minimal.
\end{enumerate}
\end{thm}
\begin{proof}
Item~\ref{preH:qeitem} follows from a standard quantifier elimination test (see for example \cite[Corollary~B.11.9]{adamtt}), using quantifier elimination for closed ordered differential fields, Corollary~\ref{adh6.3.3o}, and Lemma~\ref{qe:sameres}.

For item~\ref{preH:modcompitem}, argue as in Corollary~\ref{multimodcompa} that every pre-$H$-field with gap~$0$ extends to a model of $T^{\dhl}_{\codf}$, using the definition of closed ordered differential fields to check that they are expint-closed and linearly surjective.

The completeness in item~\ref{preH:compitem} follows from \ref{preH:qeitem} using that the ordered ring $(\Z; +,-,\cdot,0,1,\les)$ expanded by the trivial derivation and dominance relation embeds into every pre-$H$-field (or use Corollary~\ref{ake}).
Decidability then follows, since $T^{\dhl}_{\codf}$ has a recursive axiomatization.

Item~\ref{preH:distal} follows by \cite[Proposition~7.1]{acgz-distal} from \ref{preH:qeitem} and the fact that RCVF, the theory of real closed fields with a nontrivial valuation whose valuation ring is convex, is distal; in applying \cite[Proposition~7.1]{acgz-distal}, note that a derivation satisfies its assumptions, as shown implicitly in the proof of Proposition~\ref{NIPtransfer}.

Item~\ref{preH:locomin} is similar to \cite[Proposition~16.6.8]{adamtt}, which is the statement of local o-minimality of $T^{\nl}$ (strictly speaking, it is equivalent to local o-minimality by taking fractional linear transformations).
By \ref{preH:qeitem} and compactness, it suffices to show that for $L \succe K \models T^{\dhl}_{\codf}$ with $f,g>K$ in $L$, there is an isomorphism $K\langle f \rangle \to K\langle g \rangle$ over $K$, which is done as in Lemma~\ref{notcofinal}. The only difference is that now $vf<\Gamma$ instead of $\Gamma^< < vf < 0$, so use that $\Psi=\Gamma^<$ to get $vf_n<\Gamma$ for all $n$, where $f_n$ is as in the proof of Lemma~\ref{notcofinal}.
\end{proof}

\section{Examples}\label{sec:examples}

In this section, we provide two examples of pre-$H$-fields with gap~$0$.
The first, outlined in the introduction, is a model of $T^{\dhl}$, with $T_{\res} = T^{\nl}_{\sm}$, that is transexponential and whose residue field is exponentially bounded; we have continued the study of such structures in \cite{pc-transtamepair,pc-preH-dim}.
The second shows that the assumption of exponential integration is necessary in Corollary~\ref{ake} and also that $T^{\dhl}$ is satisfiable for any $T_{\res}$ extending the theory of ordered differential fields that are real closed, linearly surjective, and have exponential integration.

\subsection{A transexponential pre-\texorpdfstring{$H$}{H}-field}\label{sec:transexppreH}
This subsection elaborates on \cite[Example~10.1.7]{adamtt}.
Let $\T^*$ be an $\aleph_0$-saturated elementary extension of $\T$.
Enlarging the valuation ring $\ca O_{\T^*}$ of $\T^*$ to $\dot{\ca O}_{\T^*} = \{ f \in \T^* : |f| \les \exp^n (x)\ \text{for some}\ n \ges 1 \}$, where $\exp^n$ denotes the $n$-th iterate of the exponential function, yields a pre-$H$-field $(\T^*, \dot{\ca O}_{\T^*})$ that is $\d$-henselian, real closed, and has exponential integration.
The saturation ensures that $\dot{\ca O}_{\T^*}$ is a proper subring of $K$, i.e., $K$ contains a transexponential element.
Moreover, the residue field of $(\T^*, \dot{\ca O}_{\T^*})$ is elementarily equivalent to $\T$ as an ordered valued differential field.
To explain this, we first review the theory of $\T$.

In the language $\{+,-,\cdot, 0, 1, \der, \prece, \les\}$ the theory of $\T$ is model complete and axiomatized by the theory $T^{\nl}_{\sm}$ of newtonian, $\upomega$-free, Liouville closed $H$-fields with small derivation.
An asymptotic field $K$ is \deft{differential-valued} (\deft{$\d$-valued} for short) if $\ca O = C + \cao$; $\d$-valued fields are pre-$\d$-valued and an \deft{$H$-field} is a $\d$-valued pre-$H$-field.
An $H$-field $K$ is \deft{Liouville closed} if it is real closed, has exponential integration, and has \deft{integration} in the sense that $\der$ is surjective.
For more on $H$-fields and related notions, see \cite[Chapter~10]{adamtt}.
The property of $\upomega$-freeness is crucial to studying $\T$ but incidental here, so we refer the reader to \cite[\S11.7]{adamtt}.
Likewise, we do not define newtonianity, a technical cousin of $\d$-henselianity, and instead refer the reader to \cite[Chapter~14]{adamtt}.

To describe $(\T^*, \dot{\ca O}_{\T^*})$ and its residue field, it is convenient to work with value groups instead of valuation rings via the notions of coarsening and specialization.
Let $K$ be pre-$H$-field with small derivation and let $\Delta$ be a nontrivial proper convex subgroup of $\Gamma$ such that $\psi(\Delta^{\neq}) \subseteq \Delta$ and $\psi(\Gamma \setminus \Delta) \subseteq \Gamma \setminus \Delta$.
The \deft{coarsening of $K$ by $\Delta$} is the ordered differential field $K$ with the valuation
\begin{align*}
v_\Delta \colon K^{\x} &\to \Gamma/\Delta\\
a &\mapsto va + \Delta,
\end{align*}
denoted by $K_\Delta$.
Its valuation ring is \[\dot{\ca O}\ \coloneqq\ \{ a \in K : va \ges \delta\ \text{for some}\ \delta \in \Delta \}\ \supseteq\ \ca O\] with maximal ideal \[\dot{\cao}\ \coloneqq\ \{ a \in K : va > \Delta \}\ \subseteq\ \cao.\]
Then $K_\Delta$ is asymptotic with gap~$0$ by \cite[Corollary~9.2.26 and Lemma~9.2.24]{adamtt}, so it is pre-$\d$-valued.
Moreover, $K_\Delta$ is a pre-$H$-field: $\dot{\ca O}$ remains convex since $\dot{\cao}\subseteq\cao\subseteq(-1,1)$ and $K_{\Delta}$ still satisfies \ref{ph3} since $\dot{\ca O}\supseteq\ca O$.

Setting $\dot a \coloneqq a + \dot{\cao}$ for $a \in \dot{\ca O}$, we equip the ordered differential residue field $\dot K \coloneqq \dot{\ca O}/\dot{\cao}$ of $K_\Delta$ with the valuation
\begin{align*}
v \colon \dot{K}^{\x} &\to \Delta\\
\dot{a} &\mapsto va,
\end{align*}
making $\dot K$ an ordered valued differential field with small derivation called the \deft{specialization of $K$ to $\Delta$}.
Its valuation ring is $\ca O_{\dot K} \coloneqq \{ \dot a : a \in \ca O \}$ with maximal ideal $\cao_{\dot K} \coloneqq \{ \dot a : a \in \cao \}$.
Clearly, $\ca O_{\dot K}$ is convex.
The map $\ca O \to \ca O_{\dot K}$ given by $a \mapsto \dot{a}$ induces an isomorphism $\res(K) \cong \res(\dot{K})$ of ordered differential fields.
By \cite[Lemma~10.1.8]{adamtt}, $\dot K$ is pre-$\d$-valued and if $K$ is $\d$-valued, then so is $\dot{K}$ with $C_{\dot K}=C$, where we identify $C$ with a subfield of $C_{\dot K}$ via $a \mapsto \dot a$.
To see that $\dot K$ is a pre-$H$-field, it remains to check \ref{ph3}, so let $a \in \dot{\ca O}$ with $\dot a > \ca O_{\dot{K}}$.
Then $va \in \Delta^{<}$ and so $v(a')=va+\psi(va) \in \Delta$.
Also, $a'>0$, so $a' > \dot{\cao}$, and thus $\dot{a}'>0$.
If $K$ is an $H$-field, then so is $\dot{K}$, since $H$-fields are exactly the $\d$-valued pre-$H$-fields.

Suppose now that $K=\T^*$ and set $\Delta \coloneqq \{ \gamma \in \Gamma : \psi^n(\gamma) \ges 0\ \text{for some}\ n \ges 1 \}$, a nontrivial proper convex subgroup of $\Gamma$ with $\psi(\Delta^{\neq}) \subseteq \Delta$ and $\psi(\Gamma \setminus \Delta) \subseteq \Gamma \setminus \Delta$, as in \cite[Example~10.1.7]{adamtt}.
It is an exercise to check that $\dot{\ca O}=\dot{\ca O}_{\T^*}$ as defined above.
To summarize the preceding discussion, $K_\Delta$ is a pre-$H$-field with gap~$0$ and $\dot{K}$ is an $H$-field with $C_{\dot{K}}=C$, an elementary extension of $\R$.
As the differential field structure is unchanged, $K$ is still real closed and has exponential integration.
And since $K$ is newtonian, $K_\Delta$ is $\d$-henselian by \cite[Lemma~14.1.2]{adamtt}; this also uses that $1 \coloneqq v(1/x) > 0$ is the unique element of $\Gamma^{\neq}$ with $\psi(1)=1$ (see \cite[Lemma~9.2.15]{adamtt}), so automatically $1 \in \Delta$.

We now proceed to show that $\dot K \models T^{\nl}_{\sm}$.
First, since $K$ is $\upomega$-free, so is $\dot K$ by \cite[Lemma~11.7.24]{adamtt}.
Since $K$ is real closed, so is $\dot K$, and since $K$ has integration, so does $\dot K$ by \cite[Lemma~9.4.13]{adamtt}.
To see that $\dot K$ has exponential integration, let $a \in \dot{\ca O} \setminus \dot{\cao}$ and take $y \in K^{\x}$ with $y^\dagger=a$.
We may have $vy=0 \in \Delta$.
If $vy\neq 0$, from $\psi(vy)=v(y^\dagger)=va \in \Delta$, we get $vy \in \Delta$.
In either case, $\dot{y} \in \dot{K}^{\x}$ and $\dot{y}^\dagger = \dot a$.
Hence, $\dot K$ is Liouville closed.

It remains to see that $\dot K$ is newtonian.
To explain this, for a valued differential field $F$ and $a \in F^{\x}$, we let $F^{a}$ denote the valued field $F$ equipped with the derivation $a^{-1}\der_F$; if $F$ is newtonian, then so is $F^{a}$.
Let $\phi \in \dot{\ca O} \setminus \dot{\cao}$ be such that $\psi(\delta) \ges v\phi \in \Delta$ for some $\delta \in \Delta^{\neq}$.
Then by \cite[Lemma~14.1.4]{adamtt}, it suffices to show that $(\dot{K}^{\dot\phi}, v_{\dot\phi}^\flat)$, the coarsening of $\dot{K}^{\dot\phi}$ by $\Delta_{\dot\phi}^\flat \coloneqq \{ \delta \in \Delta : \psi(\delta) > v\dot{\phi} \}$, is $\d$-henselian.
First, note that $\Delta_{\dot\phi}^\flat \neq \0$, since $\psi(\Delta^{\neq})$ has no maximum, and that $\Delta_{\dot\phi}^\flat = \Gamma_\phi^\flat \coloneqq \{ \gamma \in \Gamma : \psi(\gamma) > v\phi \}$,
since for any $\gamma \in \Gamma_\phi^\flat$ we have $\psi(\gamma) \in \Delta$ and thus $\gamma \in \Delta$.
It follows that $(\dot{K}^{\dot\phi}, v_{\dot\phi}^\flat)$ is isomorphic to the specialization of $(K^\phi, v_\phi^\flat)$ to $\Delta/\Gamma_\phi^\flat$, where $(K^\phi, v_\phi^\flat)$ denotes the coarsening of $K^\phi$ by $\Gamma_\phi^\flat$.
Since $K$ is newtonian, so is $K^\phi$.
Hence, $(K^\phi, v_\phi^\flat)$ is $\d$-henselian by \cite[Lemma~14.1.2]{adamtt} again; in applying this lemma, note that the new $1_{\phi}>0$ is the unique element of $\Gamma^{\neq}$ with $\psi(1_{\phi})-v\phi = 1_{\phi}$, so again $1_{\phi} \in \Gamma_\phi^\flat$.
Thus the specialization of $(K^\phi, v_\phi^\flat)$ to $\Delta/\Gamma_\phi^\flat$ is also $\d$-henselian by \cite[Lemma~7.1.6]{adamtt}.

Putting this together, we have established:
\begin{prop}
Let $K \coloneqq \T^*$ be an $\aleph_0$-saturated elementary extension of $\T$ and set
$\Delta\ \coloneqq\ \{ \gamma \in \Gamma : \psi^n(\gamma) \ges 0\ \text{for some}\ n \ges 1 \}.$
Then the coarsening $K_{\Delta}$ of $K$ by $\Delta$ \textnormal{(}equivalently, $K$ equipped with the valuation ring $\dot{\ca O} = \{ f \in \T^* : |f| \les \exp^n (x)\ \text{for some}\ n \ges 1 \}$\textnormal{)} is a pre-$H$-field that is $\d$-henselian, real closed, and has exponential integration.
The residue field $\dot{K}$ of $K_{\Delta}$ is elementarily equivalent to $\T$ as an ordered valued differential field.
\end{prop}

In summary, we have thus decomposed $\T^*$ into a transexponential pre-$H$-field with gap~$0$, $K_{\Delta}$, and an exponentially bounded model of $T^{\nl}_{\sm}$, $\dot{K}$.
Combining this with the results of the previous section and facts about $T^{\nl}_{\sm}$ from \cite{adamtt}, we obtain:
\begin{cor}
Let $\ca L_{\res} = \{+,-,\cdot, 0, 1, \der, \prece, \les\}$ and $T_{\res} = T^{\nl}_{\sm}$.
Then $(K_\Delta, \dot{K}; \pi_2) \models T^{\dhl}$, where $\pi_2$ is the binary version of the residue map $K \to \dot{K}$, and $\dot{K}$ is purely stably embedded in $(K_\Delta, \dot{K}; \pi_2)$ in the sense of Corollary~\ref{stabembed}.
Also, $T^{\dhl}$:
\begin{enumerate}
    \item is complete;
    \item is model complete, and moreover is the model companion of the theory $T$ from Corollary~\ref{multimodcompa} with $T'_{\res}$ being the theory of $H$-fields with small derivation;
    \item has NIP;
    \item has quantifier elimination if $\ca L_{\res}$ and $T^{\nl}_{\sm}$ are expanded by a function symbol for multiplicative inversion and two unary predicates $\Upomega$ and $\Uplambda$ as in \cite[Chapter~16]{adamtt}.
\end{enumerate}
\end{cor}
\begin{proof}
Note that $T_{\res}$ is required to expand the theory of linearly surjective, expint-closed ordered differential fields, and indeed newtonian fields are linearly surjective by \cite[Corollary~14.2.2]{adamtt}.
Then the completeness of $T^{\dhl}$ follows from the completeness of $T^{\nl}_{\sm}$ \cite[Corollary~16.6.3]{adamtt} by Corollary~\ref{ake}.
Likewise, the model completeness and model companion statements follow from \cite[Corollary~16.2.5]{adamtt} and Corollary~\ref{multimodcompl}, 
NIP follows from \cite[Proposition~16.6.6]{adamtt} and Proposition~\ref{NIPtransfer}, and quantifier elimination from \cite[Theorem~16.0.1]{adamtt} and Theorem~\ref{relQE}.
\end{proof}

\subsection{Exponential integration is necessary}\label{sec:expintnecessary}
Unlike the AKE theorems for monotone fields, Corollary~\ref{ake} includes the assumption of closure under exponential integration.
Corollary~\ref{ex:notelem} shows that the theorem fails when this assumption is dropped by exhibiting two $\d$-henselian, real closed pre-$H$-fields that are not elementarily equivalent but have isomorphic ordered differential residue fields.
More precisely, one of these pre-$H$-fields has exponential integration but the other does not, and, indeed, the issue is pervasive in the sense that it works for any ordered differential residue field that is linearly surjective and expint-closed.
Additionally, the asymptotic couples of these pre-$H$-fields are elementarily equivalent, which shows that some assumption on logarithmic derivatives is necessary not just in Corollary~\ref{ake}, but even in any possible three-sorted extension of it.

This lemma is a variant of \cite[Lemma~10.4.6]{adamtt}.
We now set $A^{\succ 1} \coloneqq \{ a \in A : a \succ 1\}$ for $A \subseteq K$, 
and make use of the background material on asymptotic couples from \S\ref{sec:ac-small}.
\begin{lem}\label{ex:adjexpint}
Suppose that $K$ is pre-$\d$-valued of $H$-type with gap~$0$.
Let $s \in K$ such that $vs \in \Gamma^< \setminus \Psi$.
Take $a$ transcendental over $K$ with $a^\dagger = s$.
Then the differential field $K_a \coloneqq K(a^{\Q})$ can be equipped with a unique valuation making $K_a$ a pre-$\d$-valued extension of $K$ of $H$-type such that $a \succ 1$.
With this valuation, $K_a$ has gap~$0$ and satisfies:
\begin{enumerate}
    \item\label{ex:adjexpint:res} $\res(K_a)=\res(K)$;
    \item\label{ex:adjexpint:daggers} for every $f \in K_a^{\succ 1}$, either $f^\dagger \sim g^\dagger$ for some $g \in K^{\succ 1}$ or $f^\dagger \sim es$ for some $e \in \Q^{>}$;
    \item\label{ex:adjexpint:ac} the asymptotic couple $(\Gamma_a, \psi_a)$ of $K_a$ satisfies $\Gamma_a = \Gamma \oplus \Q\alpha$, with $\alpha = v(a)<0$ and $[\alpha] \notin [\Gamma]$, and $\psi_a(\gamma + e\alpha) = \min\{\psi(\gamma), v(es)\}$ for all $\gamma \in \Gamma$ and $e \in \Q$;
    \item\label{ex:adjexpint:preH} if $K$ is a pre-$H$-field with $s>0$, then $K_a$ can be equipped with a unique ordering making $K_a$ an ordered field extension of $K$ such that the valuation ring $\ca O_a$ of $K_a$ is convex; with this ordering, $K_a$ is a pre-$H$-field.
\end{enumerate}
\end{lem}
\begin{proof}
We first explain how we adjoin a new element $\alpha$ to the asymptotic couple $(\Gamma, \psi)$, which later will satisfy $\alpha=v(a)$.
Let $D_1\coloneqq\{\gamma \in \Gamma^< : \psi(\gamma)<vs\}$ and $D_2\coloneqq\{\gamma \in \Gamma^{\les} : \psi(\gamma)>vs\}$, and take $\alpha$ such that $D_1<\alpha<D_2$ and $[D_1]>[\alpha]>[D_2]$.
Recall that since $(\Gamma, \psi)$ has gap~$0$, for any $\gamma \in \Gamma^<$, we have $\gamma<\psi(\gamma)<0$.
In particular, $vs \in D_2$ and $[\alpha]>[vs]$.
Set $\Gamma_a \coloneqq \Gamma \oplus \Q\alpha$.
Let $\gamma \in \Gamma$ and $e \in \Q$, and define $\gamma + e\alpha>0$ if either $[\gamma]>[e\alpha]$ and $\gamma>0$ or $[e\alpha]>[\gamma]$ and $e<0$; if $e \neq 0$, then $\gamma + e\alpha>0$ if and only if either $[\gamma]\in[D_1]$ and $\gamma>0$ or $[\gamma]\in[D_2]$ and $e<0$.
This makes $\Gamma_a$ an ordered abelian group extending $\Gamma$ such that $\alpha<0$.
Now we extend $v \colon K^{\x} \to \Gamma$ to the unique valuation $v \colon K_a^{\x} \to \Gamma_a$ satisfying $v(a)=\alpha$.
More explicitly, every element of $K_a^{\x}$ is of the form $p/q$ for some $p=\sum_{e \in \Q} p_e a^e$ and $q=\sum_{e \in \Q} q_e a^e$, where all $p_e, q_e \in K$ with only finitely many of them nonzero and at least one $p_e$ and one $q_e$ are nonzero.
For such $p$ and $q$, we set \[v(p/q)\ \coloneqq\ \min\{v(p_e)+e\alpha : p_e \neq 0\} - \min\{v(q_e)+e\alpha : q_e \neq 0\},\]
and it is routine to check that this map is a valuation.
If $w$ is any valuation on $K_a$ making it a pre-$\d$-valued extension of $K$ of $H$-type such that $wa<0$, then $D_1<wa<D_2$ and $[D_1]>[wa]>[D_2]$, and it follows that $w = v$ after identifying their value groups via $wa \mapsto \alpha$.

For every $f \in K_a^{\x}$, there are $g \in K^{\x}$ and a unique $e \in \Q$ such that $f \sim ga^e$.
Item \ref{ex:adjexpint:res} follows.
Now we use \cite[Lemma~10.1.19]{adamtt} with $T = K^{\x}a^{\Q}$ and $L = K_a$ to show that $K_a$ is pre-$\d$-valued.
Let $t=ga^e \in T$, where $g \in K^{\x}$ and $e \in \Q$, and suppose that $t\prec 1$.
It suffices to show that $t^\dagger \succ 1$ and $t' \prec 1$.
We may assume that $e \neq 0$.
Then $v(g^\dagger) \neq vs=v(es)$, so $v(t^\dagger)=v(g^\dagger+es) = \min\{v(g^\dagger), vs\}<0$.
If $[vg]<[\alpha]$, then $e<0$ and $v(g^\dagger)>vs$, and so
\[v(t')\ =\ v(t^\dagger)+vt\ =\ vs+e\alpha+vg\ >\ 0.\]
If $[vg]>[\alpha]$, then similarly $v(t') = vg'+e\alpha > 0$.
Hence $K_a$ is pre-$\d$-valued with gap~$0$.
To prove \ref{ex:adjexpint:daggers}, let $f \in K_a^{\succ 1}$ and take $g \in K^{\x}$ and $e \in \Q$ with $f \sim ga^e$.
Then $f^\dagger \sim g^\dagger + es$, so likewise either $f^\dagger \sim g^\dagger$ and $vg<0$ or $f^\dagger \sim es$ and $e>0$.
This also finishes the proof of~\ref{ex:adjexpint:ac}.

Now we show that $(\Gamma_a, \psi_a)$ is of $H$-type.
Let $\gamma_1, \gamma_2 \in \Gamma$ and $e_1, e_2 \in \Q$ and suppose that $\gamma_1+e_1\alpha<\gamma_2+e_2\alpha<0$; we need to show that $\psi(\gamma_1+e_1\alpha) \les \psi(\gamma_2+e_2\alpha)$.
If $[\gamma_i]>[e_i\alpha]$ for $i=1,2$, then $[\gamma_1] \ges [\gamma_2]$ and $\psi(\gamma_1+e_1\alpha)=\psi(\gamma_1) \les \psi(\gamma_2)=\psi(\gamma_2+e_2\alpha)$.
If $[\gamma_i]<[e_i\alpha]$ for $i=1,2$, then $\psi(\gamma_1+e_1\alpha)=vs=\psi(\gamma_2+e_2\alpha)$.
If $[\gamma_1]>[e_1\alpha]$ and $[\gamma_2]<[e_2\alpha]$, then $\psi(\gamma_1+e_1\alpha)=\psi(\gamma_1)<vs=\psi(\gamma_2+e_2\alpha)$.
Finally, if $[\gamma_1]<[e_1\alpha]$ and $[\gamma_2]>[e_2\alpha]$, then $[e_1\alpha]>[\gamma_2]$ and $\psi(\gamma_1+e_1\alpha)=vs<\psi(\gamma_2)=\psi(\gamma_2+e_2\alpha)$ (in this case, $e_2=0$).

To show \ref{ex:adjexpint:preH}, suppose that $K$ is a pre-$H$-field and $s>0$.
Let $f \in K_a^{\x}$, and take $g \in K^{\x}$ and $e \in \Q$ with $f \sim ga^e$.
For $K_a$ to be an ordered field extension of $K$, we must have $a^e>0$ for all $e \in \Q$, and hence there is at most one way to define the ordering: $f>0$ if and only if $g>0$.
First, this is independent of the choice of $g$: if $f \sim g_1a^{e}$ with $g_1 \in K^{\x}$, then $g_1 \sim g$, so $g_1>0$ if and only if $g>0$.
To see that this makes $K_a$ an ordered field, let $f_1, f_2 \in K_a^{\x}$ with $f_1, f_2>0$.
For $i=1,2$, take $g_i \in K^{\x}$ and $e_i \in \Q$ such that $f_i \sim g_ia^{e_i}$, so $g_i>0$.
Then $f_1f_2 \sim g_1g_2a^{e_1+e_2}>0$.
Similarly, $f^2>0$ for all $f \in K_a^{\x}$.
If $f_1\prec f_2$, then $f_1+f_2 \sim f_2 > 0$; the case $f_1 \succ f_2$ is symmetric.
Suppose that $f_1 \asymp f_2$, so $e_1=e_2$ and $g_1 \asymp g_2$.
Then since $g_1, g_2>0$, we have $g_1+g_2 \asymp g_1$ and thus $f_1+f_2 \sim (g_1+g_2)a^{e_1}>0$.

To see that $\ca O_a$ is convex with respect to this ordering, we show that its maximal ideal $\cao_a$ satisfies $\cao_a \subseteq (-1,1)$:
If $f\in \cao_a$, then $1 \pm f \sim 1a^0>0$.
To show that $K_a$ is a pre-$H$-field, it remains to check that $f^\dagger>0$ whenever $f \in K_a$ satisfies $f>\ca O_a$, which follows from the proof of \ref{ex:adjexpint:daggers} and the fact that $K$ is a pre-$H$-field.
\end{proof}

\begin{prop}\label{ex:noexpint}
Let $\bm k$ be an ordered differential field that is linearly surjective and real closed.
Then there is a $\d$-henselian, real closed pre-$H$-field $K$ such that $\Psi = \Gamma^<$ and $\res(K) \cong \bm k$, but $(K^{\x})^\dagger \neq K$.
\end{prop}
\begin{proof}
We construct increasing sequences of pre-$H$-fields with gap~$0$ $(K_n)$ and sets $(\f M_n)$ satisfying, for all $n$:
\begin{enumerate}
    \item $K_n = \bm k(\f M_{n+1})$;
    \item $\f M_{n+1}$ is a divisible monomial group for $K_n$;
    \item $\f m>0$ for all $\f m \in \f M_n$;
    \item $\f M_n^{\succ 1} \subseteq (\f M_{n+1}^{\succ 1})^\dagger$;
    \item\label{noexpintiv} for every $f \in K_n^{\succ 1}$, there are $\f m \in \f M_n^{\succ 1}$ and $e \in \Q^{>}$ such that $f^\dagger \sim e\f m$.
\end{enumerate}
In particular, each $K_n$ has ordered differential residue field naturally isomorphic to $\bm k$.

Let $(b_i)_{i \in \Q}$ be algebraically independent over $\bm k$ and set $\f M_0 \coloneqq \{b_i : i \in \Q\}$.
Let $k$ range over $\N^{\ges 1} \coloneqq \{ n : n\ges 1 \}$ and $\bm i = (i_1, \dots, i_k)$ and $\bm e = (e_1, \dots, e_k)$ over $\Q^{k}$ with $i_1<\dots<i_k$.
We set $b^{\bm e}_{\bm i} \coloneqq \prod_{j=1}^{k} b_{i_j}^{e_j}$ and define $\f M_1 \coloneqq \{b_{\bm i}^{\bm e} : \bm i, \bm e \in \Q^{k}\ \text{with}\ i_1<\dots<i_k,\ k \in \N^{\ges 1}\}$ and $K_0 \coloneqq \bm k(\f M_1)$.
Set $\Gamma_0 \coloneqq \bigoplus_{i \in \Q} \Q\beta_i$ with $(\beta_i)_{i \in \Q}$ $\Q$-linearly independent, and order $\Gamma_0$ lexicographically with each $\beta_i<0$.
That is, if $e_1 \neq 0$, define $e_1\beta_{i_1} + \dots + e_k\beta_{i_k}>0$ if $e_1<0$, so $[e_1\beta_{i_1} + \dots + e_k\beta_{i_k}] = [\beta_{i_1}]$, and $[\beta_i] > [\beta_j]$ if and only if $i<j$, for $i, j \in \Q$.

We equip $K_0$ with the unique valuation $v \colon K_0^{\x} \to \Gamma_0$ satisfying $v(u)=0$ and $v(b_i)=\beta_i$ for all $u \in \bm k^{\x}$ and $i \in \Q$.
To do so more explicitly, note that every element $y \in K_0^{\x}$ has the form $y=(\sum_{\bm e} f_{\bm e}b_{\bm i}^{\bm e})/(\sum_{\bm e} g_{\bm e}b_{\bm i}^{\bm e})$ for some $\bm i$ and some $f_{\bm e}, g_{\bm e} \in \bm k$ such that only finitely many $f_{\bm e}$ and $g_{\bm e}$ are nonzero and at least one $f_{\bm e}$ and one $g_{\bm e}$ are nonzero.
Define $v \colon K_0^{\x} \to \Gamma_0$ by $v(ub_{\bm i}^{\bm e}) \coloneqq \sum_{j=1}^{k} e_{i_j} \beta_{i_j}$, for $u \in \bm k^{\x}$,
and, for $y \in K_0^{\x}$ as above
\[
v(y)\ =\ v\left(\frac{\sum_{\bm e} f_{\bm e}b_{\bm i}^{\bm e}}{\sum_{\bm e} g_{\bm e}b_{\bm i}^{\bm e}}\right)\ \coloneqq\ \min\{v(b_{\bm i}^{\bm e}) : f_{\bm e} \neq 0\} - \min\{v(b_{\bm i}^{\bm e}) : g_{\bm e} \neq 0\}.
\]
It is routine to check that this map is a valuation.
Hence $\f M_1$ is a divisible monomial group for $K_0$.
Moreover, for all $f \in K_0^{\x}$ with $f \not\asymp 1$, there exist unique $u \in \bm k^{\x}$, $\bm i$, and $\bm e \in (\Q^{\x})^k$ such that $f \sim ub_{\bm i}^{\bm e}$; for all $f \in K_0^{\x}$ with $f \asymp 1$, there exists a unique $u \in \bm k^{\x}$ such that $f \sim u$.

We extend the derivation on $\bm k$ to $K_0$ so that $(b_i^e)^\dagger = eb_{i+1}$ for all $i, e \in \Q$.
It is easy to verify that for all $t \in T \coloneqq \{ ub_{\bm i}^{\bm e} : u \in \bm{k}^{\x},\ \bm{i}, \bm{e} \in \Q^{k}\ \text{with}\ i_1<\dots<i_k,\ k \in \N^{\ges 1} \}$, if $t \prec 1$, then $t' \prec 1$ and $t^\dagger \succ 1$.
Then by \cite[Lemma~10.1.19]{adamtt} (with $K=\bm k$ and $L=K_0$), $K_0$ is pre-$\d$-valued and has gap~$0$.
Also, for every $f \in K_0^{\succ 1}$, there exist $i, e \in \Q$ with $e > 0$ and $f^\dagger \sim eb_i$.
Hence the asymptotic couple of $K_0$ is $(\Gamma_0, \psi_0)$, where $\psi_0$ is defined by $\psi_0(e_1\beta_{i_1} + \dots + e_k\beta_{i_k}) = \beta_{i_1+1}$ if $e_1 \neq 0$, and it satisfies $\Psi_0 \coloneqq \psi(\Gamma_0^{\neq}) = v(\f M_0)$.

To extend the ordering on $\bm k$ to $K_0$, let $f \in K_0^{\x}$ and take $u \in \bm k^{\x}$ and $\f m \in \f M_1$ with $f \sim u\f m$.
Then define $f>0$ if $u>0$.
The proof that this makes $K_0$ a pre-$H$-field is similar to the corresponding part of the proof of Lemma~\ref{ex:adjexpint}.

To construct $K_1$ and $\f M_2$, enumerate $(\f M_1 \setminus \f M_0)^{\succ 1}$ by $(\f m_n)$ and set $L_{0} \coloneqq K_0$ and $\f N_{0} \coloneqq \f M_{1}$.
By \ref{noexpintiv}, $v\f m_n \in \Gamma_{0}^< \setminus \Psi_{0}$ for all $n$.
Hence by repeated applications of Lemma~\ref{ex:adjexpint}, we construct increasing sequences of pre-$H$-fields with gap~$0$ $(L_{n})$, sets $(\f N_{n})$, and elements $(\f n_n)$ such that, for all $n$:
\begin{enumerate}[label=(\alph*)]
    \item $L_{n} = \bm k(\f N_{n})$;
    \item $\f N_{n}$ is a divisible monomial group for $L_{n}$;
    \item $\f N_{n+1} = \f N_{n}\f n_n^{\Q}$;
    \item $\f n_n^\dagger = \f m_n$, $\f n_n \succ 1$, and $\f n_n^e>0$ for all $e \in \Q$;
    \item for every $f \in L_{n+1}^{\succ 1}$, either $f^\dagger \sim g^\dagger$ for some $g \in L_{n}^{\succ 1}$ or $f^\dagger \sim e\f m_n$ for some $e \in \Q^{>}$;
    \item $v\f m_m \notin \Psi_{L_{n}}$ for all $m \ges n$.
\end{enumerate}
Now we set $K_1 \coloneqq \bigcup_n L_{n}$ and $\f M_2 \coloneqq \bigcup_n \f N_{n}$.
Iterating this procedure yields the desired sequences $(K_n)$ and $(\f M_n)$.
Set $K \coloneqq \bigcup_{n} K_n = \bm k(\f M)$, where $\f M \coloneqq \bigcup_{n} \f M_n$.
Then $K$ is a pre-$H$-field with gap~$0$ such that:
\begin{enumerate}[label=(\roman*)]
    \item\label{Kres} $K$ has ordered differential residue field (isomorphic to) $\bm k$;
    \item\label{Kval} $\Gamma$ is nontrivial and divisible, and satisfies $\Gamma^< = \Psi$;
    \item\label{Kdaggerchar} for every $f \in K^{\succ 1}$, there are $\f m \in \f M^{\succ 1}$ and $e \in \Q^{>}$ such that $f^\dagger \sim e\f m$.
\end{enumerate}

We have not yet needed the assumptions on $\bm k$.
Using that $\bm k$ is linearly surjective, we pass to the $\d$-henselization of $K$, an immediate pre-$H$-field extension of $K$, thereby arranging that $K$ is $\d$-henselian while preserving \ref{Kres}--\ref{Kdaggerchar}.
Then $K$ is real closed, since $\bm k$ is real closed, $\Gamma$ is divisible, and $K$ is henselian.
Finally, by \ref{Kdaggerchar} we have $u\f m \notin (K^{\x})^\dagger$ for all $u \in \bm k \setminus \Q$ and $\f m \in \f M^{\succ 1}$.
\end{proof}

By taking a $\bm k$ as above that additionally has exponential integration, we obtain:
\begin{cor}
There exists a $\d$-henselian, real closed pre-$H$-field $K$ such that $\Psi = \Gamma^<$ and $(\res(K)^{\x})^\dagger = \res(K)$, but $(K^{\x})^\dagger \neq K$.
\end{cor}

\begin{cor}\label{ex:notelem}
There exist $\d$-henselian, real closed pre-$H$-fields $K_1$ and $K_2$ such that $\res(K_1) \cong \res(K_2)$ \textup{(}as ordered differential fields\textup{)} and $(\Gamma_1, \psi_1) \equiv (\Gamma_2, \psi_2)$, but $(K_1^{\x})^\dagger \neq K_1$ and $(K_2^{\x})^\dagger = K_2$; in particular, $K_1 \not\equiv K_2$.
\end{cor}
\begin{proof}
Let $\bm k$ be an ordered differential field that is linearly surjective and expint-closed.
Then Proposition~\ref{ex:noexpint} yields a $\d$-henselian, real closed pre-$H$-field $K_1$ with ordered differential residue field isomorphic to $\bm k$, nontrivial divisible asymptotic couple $(\Gamma_1, \psi_1)$ satisfying $\Gamma_1^< = \Psi_1$, and $(K_1^{\x})^\dagger \neq K_1$.
Now let $K_2$ be the $\d$-Hensel-Liouville closure of $K_1$ by Theorem~\ref{dhlclosure}.
Then $\res(K_2)=\res(K_1)$ and the nontrivial divisible asymptotic couple $(\Gamma_2, \psi_2)$ of $K_2$ satisfies $\Gamma_2^< = \Psi_2$.
Hence $(\Gamma_1, \psi_1) \equiv (\Gamma_2, \psi_2)$ by Corollary~\ref{ac:complete}.
\end{proof}

\section*{Acknowledgements}
Thanks are due to Lou van den Dries for helpful discussions and for comments on an earlier draft of this paper, to Anton Bernshteyn and Chris Miller for suggestions on an earlier draft of this paper, and to Matthias Aschenbrenner for comments related to Proposition~\ref{NIPtransfer}. I am also grateful to the anonymous referee for carefully reading the paper and suggesting many improvements.

This research was funded in whole or in part by the Austrian Science Fund (FWF) 10.55776/ESP450. For open access purposes, the author has applied a CC BY public copyright licence to any author accepted manuscript version arising from this submission.
This material is based upon work supported by the National Science Foundation under Grant No.\ DMS-2154086.
Presentation of some of this research at DART X was supported by NSF Grant No.\ DMS-1952694.

\printbibliography

\end{document}